\documentclass[reqno,11pt]{amsart}

\usepackage{amsmath,amsthm,amssymb}
\usepackage{mathtools}
\usepackage{mathrsfs}
\usepackage[margin=1.2in,includefoot,footskip=30pt]{geometry}
\usepackage{enumitem}
\usepackage[mathlines,displaymath]{lineno}

\usepackage{xcolor}

\definecolor{electricblue}{rgb}{0.49, 0.98, 1.0}
\definecolor{electricpurple}{rgb}{0.75, 0.0, 1.0}
\definecolor{goldenyellow}{rgb}{1.0, 0.87, 0.0}
\definecolor{brightgreen}{rgb}{0.4, 1.0, 0.0}
\definecolor{pastelmagenta}{rgb}{0.96, 0.6, 0.76}
\definecolor{jade}{rgb}{0.0, 0.66, 0.42}
\definecolor{oucrimsonred}{rgb}{0.6, 0.0, 0.0}
\definecolor{darkspringgreen}{rgb}{0.09, 0.45, 0.27}
\definecolor{blue-green}{rgb}{0.0, 0.87, 0.87}
\definecolor{peach}{rgb}{1.0, 0.9, 0.71}
\definecolor{teagreen}{rgb}{0.82, 0.94, 0.75}
\definecolor{turquoise}{rgb}{0.19, 0.84, 0.78}
\definecolor{richelectricblue}{rgb}{0.03, 0.57, 0.82}
\definecolor{mediumlavendermagenta}{rgb}{0.8, 0.6, 0.8}
\definecolor{deepskyblue}{rgb}{0.0, 0.75, 1.0}
\definecolor{lightapricot}{rgb}{0.99, 0.84, 0.69}
\definecolor{lawngreen}{rgb}{0.49, 0.99, 0.0}
\definecolor{electriccyan}{rgb}{0.0, 1.0, 1.0}

\usepackage{iftex}
\ifPDFTeX
  \usepackage{hyperref}
\else\ifLuaTeX
  \usepackage{hyperref}
\else\ifXeTeX
  \usepackage{hyperref}
\else
  \usepackage[dvipdfmx]{hyperref}
\fi\fi\fi

\hypersetup{
  colorlinks=false
}

\makeatletter
\providecommand{\@secnumpunct}{\quad}
\makeatother

\DeclareMathOperator{\Aut}{Aut}
\DeclareMathOperator{\Crit}{Crit}
\DeclareMathOperator{\CV}{CV}
\DeclareMathOperator{\Comp}{Comp}
\DeclareMathOperator{\Press}{P}
\DeclareMathOperator{\id}{id}
\DeclareMathOperator{\Rat}{Rat}
\DeclareMathOperator{\lox}{Lox}

\newcommand{\calCP}{\mathcal{CP}}

\newcommand{\calS}{\mathcal{S}}

\newcommand{\sfH}{\mathsf{H}}
\newcommand{\sfHP}{\mathsf{HP}}

\newcommand{\bsK}{\boldsymbol{K}}
\newcommand{\bsk}{\boldsymbol{k}}

\newcommand{\bi}{\boldsymbol{i}}
\newcommand{\bt}{\boldsymbol{t}}
\newcommand{\C}{\mathbb{C}}
\newcommand{\N}{\mathbb{N}}

\newcommand{\R}{\mathbb{R}}
\newcommand{\Z}{\mathbb{Z}}
\newcommand{\rsp}{\widehat{\C}}

\theoremstyle{definition}
\newtheorem{thm}{Theorem}[section]
\newtheorem{dfn}[thm]{Definition}
\newtheorem{prp}[thm]{Proposition}
\newtheorem{lem}[thm]{Lemma}
\newtheorem{cor}[thm]{Corollary}
\newtheorem{rmk}[thm]{Remark}

\newtheorem{notation}[thm]{Notation}

\newtheorem{mthm}{Main Theorem}[section]

\newtheorem*{ex*}{Example}
\newtheorem*{dfn*}{dfn}
\newtheorem*{prp*}{Proposition}
\newtheorem*{thm*}{Theorem}
\newtheorem*{claim*}{Claim}
\newtheorem*{claim1*}{Claim 1}
\newtheorem*{claim2*}{Claim 2}

\newtheorem*{SA*}{Standing assumption}

\title{On the conformal preimage decay exponent of the Julia sets of rational graph-directed Markov systems}

\author{Tadashi Arimitsu}

\address{(Tadashi Arimitsu) Graduate School of Mathematics, Nagoya University,
Furocho, Chikusaku, Nagoya, 464-8602, JAPAN} 
\email{m21002u@math.nagoya-u.ac.jp}


\begin{document}
\maketitle
\pagestyle{plain}

\begin{abstract}
We define and investigate the conformal preimage decay exponent of the Julia sets of rational graph-directed Markov systems. 
We show that this exponent coincides with the difference between the topological entropy and upper sequential capacity topological pressure for the rational skew product map associated with the system $S$. 
Here, upper sequential capacity topological pressure is a slight generalisation of upper capacity topological pressure given in \cite{MR969568}.
\end{abstract}

\tableofcontents

\section{Introduction}
Let us first review rational semigroups for an introduction to rational graph-directed Markov systems. 
Let $\Rat$ denote the set of all non-constant rational maps on $\rsp$.
The set $\Rat$ is a semigroup with its binary operation being the composition of elements of $\Rat$. 
A subsemigroup of $\Rat$ is called a \textit{rational semigroup}. 
We say that a rational semigroup $G$ is \textit{finitely generated} if it is generated by a finite subset called a \textit{generator system} $\{g_{1}, \ldots, g_{n}\} \subset \Rat$, $n \in \N$.
We denote the finitely generated rational semigroup $G$ by $\langle g_{1}, \ldots, g_{n} \rangle$.
The study of dynamics generated by finitely many rational maps goes back at least to \cite{MR1232462}, although the term, rational semigroup is not used there.
A systematic treatment in the semigroup setting was subsequently given in \cite{MR1397693}.
In \cite{MR4002398}, Sumi and Watanabe introduced a rational graph-directed Markov system, which can be seen as a generalisation of rational semigroups. 
In this new class of systems, rational maps are selected according to a Markovian rule.

The next example, given in \cite[Example 1]{MR1397693}, illustrates the difficulty of studying rational semigroups as dynamical systems:
let $a \in \C$, $|a|>1$ and $G := \langle z^2, |a|^{-1}z^2 \rangle$; then, the Julia set $J(G)$ of the rational semigroup is $\{z \in \rsp \colon 1 \leq |z| \leq |a|\}$; see \cite[Example 2]{MR1232462} for a similar example. 
One readily verifies that the image of a point in $J(G)$ under some element of $G$ is mapped outside the Julia set.

In contrast, the Julia sets of rational semigroups and, more generally, of rational GDMSs satisfy a certain form of backward invariance; see Proposition \ref{55mENtyQLm}. 
Informally, preimages of subsets of the Julia sets remain contained in the Julia sets. 
The present paper primarily focuses on a quantitative study of such preimages; see subsection \ref{8Tn5I8uaaf}. 
To the best of our knowledge, this viewpoint has not been systematically explored for rational semigroups or rational GDMSs.

\subsubsection*{Notation}
We will adopt the convention that $\N:=\{n \in \Z\colon n>0\}$ and $\N_{0}:=\{n \in \Z\colon n \geq 0\}$.
For a point $z$ of the complex plane $\C$, let $|z|$ denote the \textit{modulus} of $z$.
The \textit{cardinality} of a set $\mathcal{A}$ is denoted by $\# \mathcal{A} \in \N_{0} \cup\{\infty\}$.
We write $\id$ to denote the \textit{identity map} on $\rsp$.
For a metric space $Y$ with the distance $d_{Y}$, we write $\overline{A}^{ Y}$ to denote the \textit{closure} of a subset $A \subset Y$ with respect to $d_{Y}$. 
We always equip $\rsp$ with the spherical distance $d_{\rsp}$, and we equip $\C$ with the Euclidean distance
$d_{\C}(z,w)$.
For product spaces $Y_{1}\times Y_{2}$, we use the distance
\[
d_{Y_{1}\times Y_{2}}((x_{1},x_{2}),(y_{1},y_{2})):=\max\{d_{Y_{1}}(x_{1},y_{1}),d_{Y_{2}}(x_{2},y_{2})\}.
\]

\subsubsection*{Rational graph-directed Markov systems}
We first define directed multigraphs, which determine admissible compositions of rational maps. 
Let $V$ be a finite set, each of whose elements is called a \textit{vertex}.
Let $E$ be a finite set, each of whose elements is called an \textit{edge}.
We define the functions $\bi, \bt\colon E \to V$, and the images $\bi(e)$ and $\bt(e)$ are called the \textit{initial} and \textit{terminal} vertices, respectively.
The tuple $(V, E, \bi, \bt)$ is called a \textit{directed multigraph}.
For simplicity, we will write it simply as $(V, E) := (V, E, \bi, \bt)$, and call it a \textit{graph}.

Let $\Gamma_{e} \subset \Rat$ be a non-empty subset of $\Rat$ indexed by a directed edge $e \in E$.
The triple \[
S :=(V, E,(\Gamma_{e})_{e \in E})
\]
is called a \textit{rational graph-directed Markov system} (rational GDMS). 
We write 
\[
\Gamma(S) := \bigcup_{e \in E}\Gamma_{e}.
\] 
We say that a rational GDMS is \textit{finitely generated} if the set $\Gamma(S)$ is finite.
We also write
\[
I:=\{(g,e)\in \Rat\times E : g\in \Gamma_{e}\}.
\]
For $\alpha=(g,e)\in I$ we write
\[
g_{\alpha}:=g,\qquad e_{\alpha}:=e,\qquad \bi(\alpha):=\bi(e),\qquad \bt(\alpha):=\bt(e).
\]
We define 
\[
X(S) := \left\{(\xi_n)_{n \in \N} \in I^{\N} \colon \bt(\xi_n)=\bi(\xi_{n+1}) \text{ for each } n \in \N \right\},
\]
and in particular, for $i \in V$ we define
\[
X_{i}(S) := \left\{\xi \in X(S) \colon \bi(\xi_{1}) = i\right\}.
\]
We equip the set $\Rat$ with the topology of uniform convergence and each $\Gamma_{e} \subset \Rat$ with the subspace topology of $\Rat$.
We consider the discrete topology on $E$ and endow $X(S)$ and $X_{i}(S)$, $i \in V$ with the subspace topology induced by the product topology on $(\Rat \times E)^{\N}$. 
When a rational GDMS $S$ is finitely generated, this topology on $X(S)$ is metrisable by the distance given below. 
We define the distance $d_{X(S)}(\xi, \tau)$ on $X(S)$ to be
\[
d_{X(S)}(\xi,\tau):=2^{-\min\{n \in \N:\ \xi_n\neq\tau_n\}}
\]
for any distinct $\xi, \tau \in X(S)$;
if $\xi = \tau$, then we set $d_{X(S)}(\xi, \tau) :=0$. 

For $\xi := (\xi_{1}, \xi_{2}, \ldots) \in X(S)$ and $n \in \N$, we write $\xi|n := \xi|_{n} := (\xi_{1},\ldots, \xi_{n})$.
For $n \in \N$, we set $X^{n}(S):= \{\xi|n \colon \xi \in X(S) \}$ and $X^{*}(S):=\bigcup_{n \in \N} X^{n}(S)$.
For $i \in V$ and $n \in \N$, we write $X_{i}^{n}(S) := \{\xi \in X^{n}(S): \bi(\xi_{1})=i\}$. 
Given an element $\xi \in X^{*}(S) \cup X(S)$, we write $|\xi|$ to denote the unique $n\in \N$ such that $\xi \in X^{n}(S)$ if $\xi \in X^{*}(S)$; we set $|\xi| = \infty$ if $\xi \in X(S)$.
For $i, j \in V$, we also write $X_{i,j}(S):= \{\xi \in X^{*}(S): \bi(e_{\xi,1})=i,\ \bt(e_{\xi,|\xi|})=j\}$.
Further, for an element $\xi \in X^{*}(S) \cup X(S)$ and for $k \in \N$ with $1\leq k\leq |\xi|$, we use $\xi_{k}$ or $(g_{\xi,k},e_{\xi,k})$ to denote the $k$-th component of $\xi$.
For $\xi \in X^{*}(S)$, we set 
\[
g_{\xi} := g_{\xi, |\xi|} \circ \dots \circ g_{\xi, 1} \quad \text{ and } \quad 
e_{\xi} := (e_{\xi, 1}, \dots, e_{\xi, |\xi|})
\]
and
\[
\bi(\xi) := \bi(e_{\xi,1}) \quad \text{ and }  \quad 
\bt(\xi) := \bt(e_{\xi,|\xi|}).
\]
For $\xi \in X(S)$, we also write $\bi(\xi) := \bi(e_{\xi,1})$. 
Next, we define the sets
\begin{linenomath}
\begin{align*}
&H(S) := \left\{g_{\xi} \in \Rat \colon \xi \in X^{*}(S)\right\}, \\
&H_{i}(S) := \left\{g_{\xi} \in \Rat \colon \xi \in X^{*}(S), \bi(\xi)=i\right\},\\
&H^{j}(S) := \left\{g_{\xi} \in \Rat \colon \xi \in X^{*}(S),\bt(\xi)=j\right\} \text{ and } H^{j}_{i}(S) := H_{i}(S) \cap H^{j}(S).
\end{align*}
\end{linenomath}
With these sets, we introduce several definitions that are central to our study in this paper.
We define the \textit{Fatou set} of $S$ to be 
\[
F(S) := \left\{z \in \rsp \colon \text{the set } H(S) \text { is equicontinuous at } z\right\}
\]
and the \textit{Julia set} of $S$ is defined as its complement $J(S) :=\rsp \setminus F(S)$.
Likewise, we define the following sets at each vertex $i$.
The set 
\[
F_i(S) := \left\{z \in \rsp \colon \text{the set } H_i(S) \text { is equicontinuous at } z\right\}
\] 
is called the \textit{Fatou set of $S$ at the vertex $i$}.
Its complement $J_{i}(S):=\rsp \setminus F_i(S)$ is called the \textit{Julia set of $S$ at the vertex $i$}.
We remark that $F(S) =\bigcap_{i \in V}F_{i}(S)$ and $J(S) = \bigcup_{i \in V}J_{i}(S)$ by definition.
As seen in the example above from \cite[Example 1]{MR1397693}, neither of the Julia sets $J(S)$ nor $J(G)$ possesses a property analogous to complete invariance. 
Throughout this paper, we assume that
\begin{equation}\label{sKJ4fORxv7}
J(S) \subset \C.
\end{equation}

Unlike rational semigroups, rational GDMSs consider only the compositions of rational maps prescribed by a graph. 
As a consequence, the classes of graphs influence rational GDMSs.
In view of this, we provide some definitions.
Let $\mathcal{G} := (V, E)$ be a directed multigraph. 
We define the \textit{adjacency matrix} of $\mathcal{G}$ to be the $V \times V$ matrix $\mathcal{A} := (A_{i,j})_{i, j\in V}$, where each entry is given by $A_{i, j} := \#\{e \in E \colon \bi(e)=i,\ \bt(e)=j\}$, $i, j\in V$.
Such a matrix $\mathcal{A}$ is said to be \textit{irreducible} if for any pair
$(i, j) \in V\times V$ there exists $n \in \N$ such that $(\mathcal{A}^{n})_{i, j} > 0$.
We say that a rational GDMS $S$ is \textit{irreducible} if the adjacency matrix of its graph is irreducible.

\subsubsection*{Rational skew product maps associated with rational GDMSs}
We next review \textit{fibred rational maps} used for the investigation of rational GDMSs. 
Fibred rational maps were first formulated in \cite{MR1704543, MR1785403}, where the author investigated their ergodic properties.
When fibred rational maps are associated with rational semigroups, studying the former as topological and complex dynamical systems has proven useful for understanding the fractal geometry of the latter; e.g., \cite{MR2153926, MR3003942}.
These associated fibred maps are called \textit{rational skew product maps} associated with a generator system; see \cite{MR1767945}.

Let us now formulate the rational skew product maps associated with a rational GDMS $S :=(V, E,(\Gamma_{e})_{e \in E})$.
The \textit{shift map} $\sigma \colon X(S) \circlearrowleft$ is given by $\xi=(\xi_{n})_{n \in \N} \mapsto (\xi_{n + 1})_{n \in \N}$.
We define a map $\tilde{f} \colon X(S) \times \rsp \circlearrowleft$ given by
\begin{equation}\label{ioGvTmYvy3}
\tilde{f}(\xi, z) :=(\sigma(\xi), {g_{\xi, 1}}(z))
\end{equation}
for each $(\xi, z) \in X(S) \times \rsp$. 
The map $\tilde{f}$ is called the \textit{rational skew product map associated with $S$}.
Note that the map $\tilde{f}$ is open, and it is surjective when the system is irreducible.
In this paper, we also refer to it as an \textit{associated rational skew product map}.
For $\xi \in X(S)$, we define the sets
\[
F_{\xi} := \left\{z \in \rsp\colon \text{the set } \{g_{\xi|n} \colon n \in \N\}\text { is equicontinuous at } z\right\}
\]
and $J_{\xi} := \rsp \setminus F_{\xi}$.
The sets $F_{\xi}$ and $J_{\xi}$ are called the \textit{non-autonomous Fatou set of $\xi$} and \textit{non-autonomous Julia set of $\xi$}, respectively. 
Due to the assumption $J(S) \subset \C$ given in \eqref{sKJ4fORxv7}, we have $J_{\xi} \subset J(S) \subset \C$.
Further, we define the sets 
\begin{equation}\label{8HWHM8kJ4p}
J(\tilde{f}) := \overline{\bigcup_{\xi \in X(S)} \{\xi\} \times J_{\xi}}^{X(S) \times \C}
\quad \text{and} \quad F(\tilde{f}):= (X(S) \times \rsp) \setminus J(\tilde{f}),
\end{equation}
which are called the \textit{skew product Julia set} of $\tilde{f}$ and the \textit{skew product Fatou set}, respectively.
The set $J(\tilde{f})$ is completely invariant under $\tilde{f}$; that is,
\begin{equation}\label{l0mSj4Qp58}
\tilde{f}^{-1}(J(\tilde{f}))=\tilde{f}(J(\tilde{f}))=J(\tilde{f}).
\end{equation}
For $i \in V$, we also define the \textit{skew product Julia set} at $i \in V$ as
\begin{equation}\label{h6uKK2bqKG}
J_{i}(\tilde{f}) := \overline{\bigcup_{\xi \in X_{i}(S)} \{\xi\} \times J_{\xi}}^{X(S) \times \C}.
\end{equation}
For $\tilde{z} := (\xi, z) \in X(S) \times \rsp$ and $n \in \N$, we set 
\[
(\tilde{f}^{n})^{\prime}(\tilde{z}) := (g_{\xi|n})^{\prime}(z).
\]
We endow $J(\tilde{f})$ and $J_{i}(\tilde{f})$ with the subspace topology induced by the product distance $d_{X(S)\times \C}$ on $X(S)\times \C$. 
Unless otherwise stated, all derivatives and their norms are taken with respect to the Euclidean coordinate on $\C$.

Let $\pi_{1} \colon X(S)\times \rsp\to X(S)$ and $\pi_{2} \colon X(S)\times\rsp\to\rsp$ denote the coordinate projections.
Under the canonical identification $\pi_{2}|_{\pi_{1}^{-1}(\{\xi\})} \colon \pi_{1}^{-1}(\{\xi\})\to\rsp$,
each fibre $\pi_{1}^{-1}(\{\xi\})$ is naturally identified with $\rsp$.
However, by \eqref{sKJ4fORxv7} we have $J_{\xi}\subset J(S)\subset\C$ for all $\xi \in X(S)$ and thus $J(\tilde{f}) \subset X(S)\times\C$.

Following \cite[Definition 2.6, Definition 2.7]{MR1827119}, we now introduce the notions about important classes of associated rational skew product maps.
The associated rational skew product map $\tilde{f} \colon X(S) \times \rsp \circlearrowleft$ is said to be \textit{expanding along fibres} if $J(\tilde{f}) \neq \varnothing$ and there exist $c > 0$ and $\lambda>1$ such that for each $n \in \N$
\begin{equation}\label{3co5RIbc1p}
\inf _{\tilde{z} \in J(\tilde{f})}|(\tilde{f}^{n})^{\prime}(\tilde{z})| \geq c \lambda^n. 
\end{equation}

When the rational skew product map associated with a rational semigroup is expanding in the sense of \cite{MR1767945}, it is a special case of open, distance-expanding maps, which have been studied in great detail; see \cite{MR2656475, MR4510391, MR4486444}. 
This is also the case for rational skew product maps associated with finitely generated rational GDMSs being expanding along fibres; see \cite{arimitsu2024}.

For $g\in \Rat$, let $\Crit(g)\subset\rsp$ denote the set of critical points of $g$, i.e. the set of points where $g$ is not locally injective.
We write $\CV(g):=g(\Crit(g))$ for the set of critical values of $g$.
In particular, if $z\in \C$ and $g(z)\in \C$, then $z\in \Crit(g)$ is equivalent to $g^{\prime}(z)=0$ in the Euclidean coordinate on $\C$. 
Write $C(\tilde{f}):=\{(\xi, z)\in X(S) \times \rsp \colon  z \in \Crit(g_{\xi,1})\}$.
We now define 
\[
P(\tilde{f}) := \overline{\bigcup_{n \in \N} \tilde{f}^{n}(C(\tilde{f}))}^{ X(S) \times \rsp}.
\]
The set $P(\tilde{f})$ is called the \textit{post-critical set} for the map $\tilde{f}$. 
We say that the map $\tilde{f}$ is \textit{hyperbolic along fibres} if {$P(\tilde{f}) \subset F(\tilde{f})$}.
For each vertex $i \in V$, the \textit{post-critical set of ${S}$ at the vertex $i$} is defined by
\[
P_{i}(S) :=\overline{\bigcup_{h \in H^{i}(S)} \CV(h)}^{ \rsp},
\]
where $\CV(h)$ denotes the set of critical values of $h \in H(S)$. 
A rational GDMS $S$ is said to be \textit{hyperbolic} if $P_{i}(S) \subset F_{i}(S)$ for each $i \in V$.

Denote by $\Aut(\rsp)$ the set of M\"obius transformations of $\rsp$ and denote by $\lox(\rsp)$ the set of loxodromic elements of $\Aut(\rsp)$.
We say that a non-identity element $g$ in $\Aut(\rsp)$ is \textit{loxodromic} if $g$ has two fixed points whose multipliers have moduli different from one. 
For an element $g \in \Rat$, we write $\deg(g) \in \N$ to refer to its \textit{degree}.
\begin{prp}[\cite{arimitsu2024}]\label{jSk7Eo33NCD} 
Let $S:= (V, E,(\Gamma_{e})_{e \in E})$ be a finitely generated, irreducible rational GDMS.
\begin{enumerate}[label=(\roman*)]
\item \label{FLr4liUUuq1} 
If the map $\tilde{f}$ is expanding along fibres, then it is hyperbolic along fibres, and moreover $\bigcup_{i \in V} H_{i}^{i}(S) \cap \Aut(\rsp) \subset \lox(\rsp)$. 
\item \label{FLr4liUUuq2} Suppose that there exists $g \in \Gamma(S)$ with $\deg(g) \geq 2$ and $\bigcup_{i \in V} H_{i}^{i}(S) \cap \Aut(\rsp) \subset \lox(\rsp)$.
If the map $\tilde{f}$ is hyperbolic along fibres, then it is expanding along fibres.
\end{enumerate}
\end{prp}

\begin{rmk}
Let $S:= (V, E,(\Gamma_{e})_{e \in E})$ be a finitely generated, irreducible rational GDMS. 
The map $\tilde{f}$ is hyperbolic along fibres if and only if the system $S$ is hyperbolic.
\end{rmk}

\subsection{Statement of results}\label{8Tn5I8uaaf}

Up until Main Theorem \ref{GQc0T8Xqyu}, we fix a finitely generated, irreducible rational GDMS denoted by $S=(V, E, (\Gamma_{e})_{e \in E})$.
Recall that
\[
I:=\{(g,e)\in \Rat\times E : g\in \Gamma_{e}\}.
\]
For $\alpha=(g,e)\in I$ we write
\[
g_{\alpha}:=g,\quad e_{\alpha}:=e,\quad
\bi(\alpha):=\bi(e) \quad \text{ and } \quad \bt(\alpha):=\bt(e).
\]

For $i, j \in V$, we define
\[
I_{(i)} := \{ \alpha \in I \colon \bi(\alpha) = i\}, \quad I^{(j)} := \{ \alpha \in I \colon \bt(\alpha) = j\} \text{ and }I_{(i)}^{(j)} := I_{(i)} \cap I^{(j)}.
\]

\begin{dfn}
Let $S=(V, E, (\Gamma_{e})_{e \in E})$ and let $I$ be the index set of generators as above.
A family $a=(a_{\alpha})_{\alpha \in I} \in [0,\infty)^I$ is called a \textit{(vertex-wise) weight}
if $\sum_{\alpha \in I^{(j)}} a_{\alpha} = 1$ for every $j\in V$.
We denote the set of all weights by $\mathcal{W}(S)$.
We write $\mathcal{W}^{+}(S):=\left\{a \in \mathcal{W}(S): \forall \alpha \in I  (a_{\alpha}>0) \right\}$.
\end{dfn}

Let $t\in \R$. 
We define the $\#V\times \#V$ matrix
$M^{(t)} = (M^{(t)}_{ij})_{i,j \in V}$ given by
\[
M^{(t)}_{ij}
:=
\sum_{\alpha \in I_{(i)}^{(j)}}
\deg(g_{\alpha})^{t}
\]
for $i,j \in V$.
We call $M^{(t)}$ the \textit{rational degree matrix} of the system $S$ (with inverse temperature $t$). 
When $t = 1$, we simply write $M=(M^{(1)}_{ij})_{i,j \in V}$.
Let 
\[
\rho_{t} = \rho(M^{(t)})
\]
denote the spectral radius of $M^{(t)}$.

Let $M^{(t)}$ be the rational degree matrix of $S$ and
let $\rho_{t}>0$ and $u^{(t)}=(u^{(t)}_i)_{i \in V} \in(0,\infty)^V$
be a left Perron-Frobenius eigenpair, $(u^{(t)})^{ \top}M^{(t)}
= \rho_{t} (u^{(t)})^{ \top}$.
For $\alpha \in I$, we define
\[
a^{(t)}_{\alpha}
 :=
 \frac{\deg(g_{\alpha})^t u^{(t)}_{\bi(\alpha)}}
{\rho_{t} u^{(t)}_{\bt(\alpha)}} \;>\;0.
\]
Note that $(a^{(t)}_{\alpha})_{\alpha \in I}$ belongs to $\mathcal{W}^{+}(S)$.
Let $h_{\mathrm{top}}(\tilde{f})$ denote the topological entropy of $\tilde{f}|_{J(\tilde{f})}$.

The author of \cite{MR1827119} proved the following fact: the topological entropy of the rational skew product map associated with a generator system $\{g_{1}, \ldots, g_{n}\}$ coincides with $\log (\sum_{j=1}^n \deg(g_j))$; see \cite[Theorem 6.10]{MR1827119}.

\begin{mthm}[see Theorem \ref{NBnbDkLhPN}]\label{GQc0T8Xqyu}
Let $S=(V, E, (\Gamma_{e})_{e \in E})$ be a finitely generated, irreducible rational GDMS.
Assume that the associated skew product $\tilde{f}$ is expanding along fibres.
Then, we have
\[
h_{\mathrm{top}}(\tilde{f}) 
=
\log\rho(M).
\]
\end{mthm}

Let $Y$ be a compact metric space with a distance $d_{Y}$ and let $T \colon Y\to Y$ be continuous.
We call the pair $(Y, T)$ a \textit{(topological) dynamical system}. 
For $n \in \N$, we define the \textit{Bowen metric} 
$d_{n}(x,y):=\max_{0\leq k<n} d_Y\left(T^k x, T^k y\right)$ and \textit{Bowen ball}  $B_{n}(x,\varepsilon):=\{y\in Y:\ d_{n}(x,y)<\varepsilon\}$.
For a real-valued function $\phi$ on $Y$, we write $\calS_{n}\phi(y):=\sum_{k=0}^{n-1}\phi(T^k y)$ for $y \in Y$.
Let $C(Y)$ denote the set of real-valued continuous functions on $Y$.

Up until Main Theorem \ref{mthm:B}, we fix a rational GDMS of the following class.
Let $S=(V, E, (\Gamma_{e})_{e \in E})$ be a finitely generated, irreducible rational GDMS.
Assume that the map $\tilde{f}$ is expanding along fibres.
Define a map $\tilde{\varphi} \colon J(\tilde{f}) \to \R$, 
\begin{equation*}
\tilde{\varphi}(\tilde{z}):= -\log |(\tilde{f})^{\prime}(\tilde{z})|
\end{equation*}
for $\tilde{z} := (\xi, z) \in J(\tilde{f})$.
For $\phi \in C(Y)$, let $\Press (T, \phi)$ denote the \textit{topological pressure} of $(T, \phi)$; refer to \cite{MR648108} for example.
For $u \in \R$, we write $\tilde\varphi_{u} := u\tilde\varphi$.
Moreover, we define a \textit{topological pressure function} $\mathcal{P} \colon \R \to \R\cup \{\infty\}$ given by 
\[
\mathcal{P} (u) :=\Press \left(\tilde{f}|_{J(\tilde{f})}, \tilde\varphi_{u}\right).
\]
Let $\delta \in \R$ denote the Bowen parameter of $S$, the unique zero of the function $\mathcal{P}$; see \cite{arimitsu2024}.
We denote $\tilde{\pi}_{2} := \pi_{2}|_{J(\tilde{f})}$ and $\tilde{\pi}_{2,i} := \pi_{2}|_{J_{i}(\tilde{f})}$ for $i \in V$. 
When the map $\tilde{f}$ is topologically transitive, it is known that there exists $\delta$-conformal measure on $J(\tilde{f})$ given in  Proposition \ref{qQrOdN2Nmt}. 
Let $\nu_{\delta}$ denote the push-forward measure of $\tilde\nu_{\delta}$ under $\tilde{\pi}_{2}$ given by 
\[
\nu_{\delta} (B) := \tilde\nu_{\delta}(\tilde{\pi}_{2}^{-1}(B))
\] 
for each Borel set $B \subset J(S)$.

We now note that the system $S$ is hyperbolic as its associated rational skew product map is expanding along fibres.
Let $D(z, r)$ denote the \textit{Euclidean disc} with centre $z \in \C$ and radius $r >0$.
Consider a family $(D(y_{i}, R))_{i \in V}$ with radius $R >0$ such that
$y_{i} \in J_{i}(S)$ and $\overline{D(y_{i},2R)}^{ \C}\cap P_{i}(S)=\varnothing$ for each $i \in V$.
For $i \in V$, we define a \textit{vertex-wise hole} for $S$ of radius $R$ by 
\begin{equation}\label{ss4cNJda4m}
\sfH_{i}(R):=D(y_{i}, R)\cap J_{i}(S).
\end{equation}

\begin{dfn}
Assume that $S$ is hyperbolic.
Let $R>0$. A family $(\sfH_{i}(R))_{i \in V}$ is called a \textit{family of vertex-wise holes of radius $R$}
if there exist points $y_{i}\in J_{i}(S)$ such that
$\overline{D(y_{i},2R)}^{ \C}\cap P_{i}(S)=\varnothing$ and
\[
\sfH_{i}(R)=D(y_{i},R)\cap J_{i}(S)
\]
for all $i \in V$.
\end{dfn}

The condition $\overline{D(y_{i},2R)}^{ \C}\cap P_{i}(S)=\varnothing$ guarantees that the relevant inverse branches
are well-defined on $D(y_{i},2R)$.
For the rest of this section, we additionally fix vertex-wise holes $(\sfH_{i}(R))_{i \in V}$. 
For $n\ge1$, we define the \textit{$n$-hole preimage} by
\begin{equation}\label{w0bkugM7X5}
\sfHP_{n}(R):=
\bigcup_{\xi \in X^{n}(S)} g_{\xi}^{-1}(\sfH_{\bt(\xi)}(R)).
\end{equation}
Note that $\sfHP_{n}(R) \subset J(S)$ by Proposition \ref{55mENtyQLm}. 
Next, we define the \textit{conformal preimage decay exponent} by
\[
e_{\delta}(R)
:=
-\limsup_{n \to \infty}\frac{1}{n}\log \frac{\nu_{\delta}(\sfHP_{n}(R))}{N_n},
\]
where $N_n:=\sum_{\xi \in X^{n}(S)}\deg(g_{\xi})$ for $n \in \N$.
We define the \textit{lifted hole} on the skew product Julia set, 
\[
\widetilde{\sfH}(R)
:=\left(\bigcup_{i \in V} (X_{i}(S)\times \sfH_{i}(R))\right)\cap J(\tilde{f}).
\]

\begin{dfn}
For $n\in \N$, we define the \textit{$n$-lifted hole preimage} by
\begin{equation}\label{sqBBlOyDPn}
\widetilde{\sfHP}_{n}(R):=\tilde{f}^{-n} \left(\widetilde{\sfH}(R)\right).
\end{equation}
For convenience, we write $\widetilde{\sfHP}(R):=(\widetilde{\sfHP}_{n}(R))_{n \in \N}$.

\end{dfn}

\begin{rmk}
Escape rates and survivor sets for open dynamical systems have been studied by many authors; see e.g. \cite{MR534126, MR2199394, MR2535206, MR2955314, MR2995652}.
We emphasise that our main objects $\sfHP_{n}(R)$, $\widetilde{\sfHP}_{n}(R)$ and $e_{\delta}(R)$ are tailored to the GDMS setting and, in particular, involve an averaging/normalisation by $N_n$.
To the best of our knowledge, we could not find a counterpart treated explicitly in the existing literature on open-systems.
\end{rmk}

We slightly generalise the capacity topological pressure defined in \cite{MR969568}.
This notion was originally designed to study subsets which are not necessarily invariant or closed; see also \cite{MR1489237}.

\begin{dfn}[\cite{MR969568}]
For $Z\subset Y$, $\phi \in C(Y)$ and $\varepsilon>0$, let $\mathcal{C}_n(Z,\varepsilon)$ be the collection of at most countable sets
$E\subset Y$ such that
\[
Z\subset \bigcup_{x\in E} B_{n}(x,\varepsilon).
\]
For a Borel set $Z \subset Y$, $\phi \in C(Y)$ and $\varepsilon>0$, we define
\[
\Lambda_n(Z,\phi,\varepsilon)
:=\inf_{E \in \mathcal{C}_n(Z,\varepsilon)}
\sum_{x\in E}\exp\left(\sup_{y\in B_{n}(x,\varepsilon)}\calS_{n}\phi(y)\right).
\]
The \textit{upper} and  \textit{lower capacity topological pressure at scale $\varepsilon$} are defined as 
\[
\overline{\calCP}(T, \phi, Z, \varepsilon)
:=
\limsup_{n\to\infty}\frac{1}{n}\log \Lambda_n(Z,\phi,\varepsilon) \quad \text{and} \quad \underline{\calCP}(T,\phi,Z,\varepsilon)
:=
\liminf_{n\to\infty}\frac1n\log\Lambda_n(Z,\phi,\varepsilon). 
\]
The \textit{upper} and \textit{lower capacity topological pressure} are defined to be 
\[
\overline{\calCP}(T, \phi, Z)
:=
\lim_{\varepsilon\to0}\overline{\calCP}(T, \phi, Z, \varepsilon)
\quad \text{and} \quad 
\underline{\calCP}(T,\phi,Z)
:=
\lim_{\varepsilon\to0}\underline{\calCP}(T,\phi,Z,\varepsilon).
\]
\end{dfn}

\begin{dfn}\label{BphPCB0UAa}
Fix $\phi \in C(Y)$, $\varepsilon>0$, and a sequence of Borel sets
$\mathcal{Z} := (Z_n)_{n \in \N}$ in $Y$.
Define the \textit{upper} and \textit{lower sequential capacity topological pressure at scale $\varepsilon$} by
\[
\overline{\calCP}(T, \phi, \mathcal{Z}, \varepsilon)
:=
\limsup_{n\to\infty}\frac1n\log \Lambda_n(Z_n,\phi,\varepsilon) \quad \text{and} \quad
\underline{\calCP}(T, \phi, \mathcal{Z}, \varepsilon)
:=
\liminf_{n\to\infty}\frac1n\log \Lambda_n(Z_n,\phi,\varepsilon).
\]
Define the \textit{upper} and \textit{lower sequential capacity topological pressure} by
\[
\overline{\calCP}(T, \phi, \mathcal{Z})
:=
\lim_{\varepsilon\to0}\overline{\calCP}(T, \phi, \mathcal{Z}, \varepsilon)  \quad \text{and} \quad
\underline{\calCP}(T, \phi, \mathcal{Z})
:=
\lim_{\varepsilon\to0}\underline{\calCP}(T, \phi, \mathcal{Z}, \varepsilon).
\]
%We say that the sequential capacity topological pressure \textit{exists} if
%$\overline{\calCP}(T, \phi, \mathcal{Z})
%=
%\underline{\calCP}(T, \phi, \mathcal{Z})$,
%and in that case we denote the common value by
%$\calCP(T, \phi, \mathcal{Z})$.
\end{dfn}

We say a rational GDMS $S$ satisfies the \textit{vertex-wise separation condition} (VSC) if 
\[
g_{1}^{-1}\left(J_{\bt(e_{1})}(S)\right) \cap g_{2}^{-1}\left(J_{\bt(e_{2})}(S)\right)=\varnothing
\]
for {each} $i \in V$ and for all $e_{1}, e_{2} \in E$ with $\bi(e_{1})=\bi(e_{2})=i$ and $g_{1} \in \Gamma_{e_{1}}, g_{2} \in \Gamma_{e_{2}}$ with {$(g_{1}, e_{1}) \neq (g_{2}, e_{2})$}. 
In the special case where a finitely generated rational GDMS is a rational semigroup $G = \langle f_{1}, \ldots, f_{n} \rangle$, the VSC is the separation condition given in \cite[Lemma 7.1]{MR1767945}. 
The VSC was originally introduced in \cite{MR4002398} and named as the backward separating condition.

\begin{mthm}[see Propositions \ref{vGhPS8pfoC}, \ref{grxaIp1jSL} and  Theorem \ref{RspvxUCucI}]
\label{mthm:B}
Let $S=(V, E, (\Gamma_{e})_{e \in E})$ be a finitely generated, irreducible rational GDMS.
Assume that the system $S$ satisfies the VSC and that the associated skew product $\tilde{f}$ is expanding along fibres
and topologically transitive on $J(\tilde{f})$.
Take an arbitrary family of vertex-wise holes $(\sfH_{i}(R))_{i \in V}$ for $S$.
Then, we have 
\[
e_{\delta}(R) = h_{\mathrm{top}}(\tilde{f}) - \overline{\calCP}\left(\tilde{f}|_{J(\tilde{f})},\delta\tilde{\varphi}, \widetilde{\sfHP}(R)\right)
> 0.
\]
\end{mthm}

\section{Preliminaries}
In this section, we collect relevant facts from previous studies mostly about rational GDMSs.

\begin{prp}[\cite{MR4002398}, Lemma 2.15, Proposition 2.16]\label{55mENtyQLm}
Let $S=(V, E, (\Gamma_{e})_{e \in E})$ be a finitely generated rational GDMS.
For each $e \in E$, we have 
\[
\bigcup_{g \in \Gamma_{e}}g^{-1}(J_{\bt(e)}(S)) = J_{\bi(e)}(S).
\]
\end{prp}

\begin{dfn}
For $\xi \in X^{*}(S)$, we define 
$
[\xi]:=\left\{\tau \in X(S) \colon \tau|_{|\xi|} = \xi \right\},
$ 
which we call the \textit{cylinder set} of $\xi$. 
For two elements $\xi := (\xi_{1}, \xi_{2}, \ldots, \xi_{|\xi|} ) \in X^{*}(S)$ and $\tau \in X^{*}(S) \cup X(S)$ with $\bt(\xi) = \bi(\tau)$, their \textit{{concatenation}} is denoted by $\xi\tau := (\xi_{1}, \cdots, \xi_{|\xi|}, \tau_{1}, \tau_{2}, \ldots) \in X^{*}(S) \cup X(S)$.
We write $\xi^{n}$ for the concatenation of $\xi \in X^{*}(S)$ with itself $n$ times when $\bt(\xi) = \bi(\xi)$.
For $\xi \in X^{*}(S)$, let $\xi^{\infty}$ denote $\tau \in X(S)$ such that for each $n \in \N$ we have $\xi^{n} = \tau|_{n|\xi|}$.
\end{dfn}

The next proposition includes important facts on the dynamics on various Julia sets.
\begin{prp}\label{HSecY02Qdj}
Let $S := (V, E,(\Gamma_{e})_{e \in E})$ be a rational GDMS. 
\begin{enumerate}[label=(\roman*)]
\item\label{nqXaAN56hT}
For each $i \in V$ and $\xi \in X_{i}(S)$, we have $J_{\xi} \subset \bigcap_{n \in \N} g_{\xi|n}^{-1}(J_{\bt(\xi|n)}(S))$.
\item\label{sc15xYijnN}  For each $\xi \in X(S)$, we have $J_{\xi}=g_{\xi,1}^{-1}(J_{\sigma(\xi)})$.
\end{enumerate}
\end{prp}

The first statement was given in \cite[Lemma 2.27]{MR4002398}.
The second was given in \cite[Lemma 2.4]{MR1827119} and restated in \cite[Lemma 2.30]{MR4002398}.

\begin{prp}[\cite{arimitsu2024}]\label{oK0KWDG1uX}
Suppose that for a finitely generated, irreducible, non-elementary rational GDMS $S$, the map $\tilde{f}$ is expanding along fibres. 
Then, the map $\tilde{f}$ is topologically transitive on $J(\tilde{f})$.
\end{prp}

\begin{prp}\label{LPtX8cirE5}
Let $S := (V, E,(\Gamma_{e})_{e \in E})$ be a finitely generated, irreducible rational GDMS. 
Then, we have $\pi_{2} (J_{i}(\tilde{f}))=J_{i}(S)$ for each $i \in V$, and thus $\pi_{2}(J(\tilde{f})) = J(S)$.
\end{prp}

The following proposition is a special case of \cite[Theorem 1.18]{MR2237476} and will be used occasionally, combined with Proposition \ref{LPtX8cirE5}.

\begin{prp}[\cite{arimitsu2024}]\label{wPJN4oayEL}
Let $S := (V, E,(\Gamma_{e})_{e \in E})$ be a rational GDMS. 
Suppose the associated rational skew product map $\tilde{f}: X(S)\times \rsp \circlearrowleft$ is expanding along fibres.
Then, we have $J(\tilde{f}) = \bigcup_{\xi \in X(S)}\{\xi\} \times J_{\xi}$. 
\end{prp}

The next claim follows from Propositions \ref{LPtX8cirE5} and \ref{wPJN4oayEL}.
\begin{cor}\label{wPJN4oayEM}
Suppose a finitely generated rational GDMS $S := (V, E,(\Gamma_{e})_{e \in E})$ is irreducible and the associated rational skew product map $\tilde{f} \colon X(S)\times \rsp \circlearrowleft$ is expanding along fibres. 
Then, we have $J_{i}(S) = \bigcup_{\xi \in X_{i}(S)}J_{\xi}$ for each $i \in V$. 
\end{cor}

We use the Koebe distortion theorem of following form.
\begin{prp}[Theorem 2.7, \cite{MR4887995}]\label{ekoebe1}
There exists a function $\bsk \colon [0,1) \to[1, \infty)$ such that for any $z \in \C$, $r>0$, $t \in [0,1)$ and any univalent analytic function $g \colon D(z, r) \to \C$, we have 
\[
\sup \left\{\left|g^{\prime}(w)\right| \colon w \in D(z, t r)\right\} \leq \bsk(t) \inf \left\{\left|g^{\prime}(w)\right| \colon w \in D(z, t r)\right\} .
\]
\end{prp}

\begin{prp}[Theorem 2.8, \cite{MR4887995}]\label{ekoebe2}
Suppose that $U \subset \C$ is an open set, $z \in U$, and $g \colon U \to \C$ is an analytic map which has an analytic inverse $g_{z}^{-1}$ defined on $D(g(z), 2 R)$ for some $R>0$. 
Then for every $0 \leq r \leq R$,
\[
D\left(z, \bsK^{-1} r\left|g^{\prime}(z)\right|^{-1}\right)
\subset g_{z}^{-1}\left(D(g(z), r)\right)
\subset D\left(z, \bsK r\left|g^{\prime}(z)\right|^{-1}\right).
\]
Here, we put $\bsK := \bsk(1/2)$.
\end{prp}

For a dynamical system $(Y, T)$, let $\mathcal{M}_{<\infty}(Y)$ denote the space of finite Borel measures on $Y$, let $\mathcal{M}(Y)$ the space of Borel probability measures on $Y$ and $\mathcal{M}^{T}(Y)$ the set of $T$-invariant Borel probability measures on $Y$.

For the rest of this section, we assume for simplicity that a rational GDMS $S$ is finitely generated and irreducible and that the map $\tilde{f}$ is expanding along fibres. 
Let $C(J(\tilde{f}))$ be the Banach space of continuous real-valued functions on $J(\tilde{f})$ with the supremum norm $| \cdot |_{\infty}$.

\begin{dfn}\label{EN2fuvjEs1}
Let $\phi \colon J(\tilde{f}) \to \R$ be a continuous function. 
The \textit{Perron-Frobenius operator} $\mathcal{L}_{\phi} \colon C(J(\tilde{f})) \circlearrowleft$ for $(\tilde{f}, \phi)$ is defined by
\begin{equation}\label{oQDCOWrEuC}
\mathcal{L}_{\phi} \psi(\tilde{y}) := \sum_{\tilde{z} \in \tilde{f}^{-1}(\tilde{y})} e^{\phi(\tilde{z})}\psi(\tilde{z})
\end{equation}
for any $\psi \in C(J(\tilde{f}))$ and $\tilde{y} \in J(\tilde{f})$.
\end{dfn}

Note that $\tilde{f}|_{J(\tilde{f})}$ is an open, distance-expanding map.
For each potential $\phi \in C(J(\tilde{f}))$, the positive linear operator
$\mathcal{L}_{\phi}$ is well-defined and bounded; see \cite[Lemma 13.6.1]{MR4486444}.
The dual operator $\mathcal{L}_{\phi}^{*} \colon \mathcal{M}_{<\infty}(J(\tilde{f})) \to \mathcal{M}_{<\infty}(J(\tilde{f}))$, given by $\int \psi  \mathrm{d}(\mathcal{L}_{\phi}^{*} m) := \int (\mathcal{L}_{\phi} \psi)  \mathrm{d}m$
is well-defined; see \cite[Theorem 6.2 and 6.3]{MR648108}. 
The next proposition follows from \cite[Theorem 13.6.2, Proposition 13.6.14, Lemma 13.6.15, Proposition 13.6.16]{MR4486444}.
\begin{prp}\label{qQrOdN2Nmt}
For each $u \in \R$, there exist $\mathfrak{c} > 0$ and $\tilde{\nu}_{u} \in \mathcal{M}_{<\infty}(J(\tilde{f}))$ such that $\mathcal{L}_{\tilde{\varphi}_{u}}^{*}\tilde{\nu}_{u} =\mathfrak{c}\tilde{\nu}_{u}$.
If, in particular, the map $\tilde{f}$ is topologically transitive on $J(\tilde{f})$, then there exists a unique zero $\delta := \delta(S)$ of the topological pressure function $\mathcal{P}$ and we have $\mathfrak{c} = e^{\mathcal{P} (u)}$.
\end{prp}

The following proposition is a special case of \cite[Corollary 13.6.11]{MR4486444}.
%%%%% Proposition %%%%%
\begin{prp}\label{tg3o75180I5}
Suppose that the map $\tilde{f}$ is topologically transitive on $J(\tilde{f})$. 
Let $n \in \N$ and let $A$ be a Borel subset of $J(\tilde{f})$ such that
$\tilde{f}^{n}|_{A}$ is injective. 
Then, we have 
\[
\tilde{\nu}_{\delta}(\tilde{f}^{n}(A))
= \int_{A} \exp({-\calS_{n}\tilde{\varphi}_{\delta}(\tilde{z})})\mathrm{d}\tilde{\nu}_{\delta}(\tilde{z})
= \int_{A} |(\tilde{f}^{n})^{\prime}(\tilde{z})|^{\delta}\mathrm{d}\tilde{\nu}_{\delta}(\tilde{z}).
\]
\end{prp}
In particular, for $n=1$,the measure $\tilde{\nu}_{\delta}$ is $e^{-\tilde{\varphi}_{\delta}}$-conformal. 

\section{Topological entropy of rational skew product maps}
The objective of this section is to prove Main Theorem \ref{GQc0T8Xqyu}.
We make the following standing assumption throughout this section.
\begin{SA*} 
Let $S=(V, E, (\Gamma_{e})_{e \in E})$ be a finitely generated, irreducible rational GDMS.
Assume that the associated skew product $\tilde{f}$ is expanding along fibres.
\end{SA*}
Under this assumption, the map $\tilde{f}|_{J(\tilde{f})}$ is an open distance-expanding map and surjective. 
Instead of repeating the argument developed in Sections 4 and 5 of \cite{MR1767945}, we utilise results on open, distance-expanding dynamical systems given in \cite{MR2656475}.

For $a=(a_{\alpha})_{\alpha \in I} \in \mathcal{W}^{+}(S)$, we define a continuous function $\phi_{a} \colon J(\tilde{f})\to\R$ by
\begin{equation*}
\phi_{a}(\xi, z)
:=
\log a_{\xi_{1}}-\log \deg(g_{\xi_{1}})
\end{equation*}
for $(\xi, z) \in J(\tilde{f})$. 
For simplicity, we write $\mathcal{L}_{a} := \mathcal{L}_{\phi_{a}}$. 
Note that for $\varphi \in C(J(\tilde{f}))$, we have 
\[
(\mathcal{L}_{\phi_{a}}\varphi)(\xi, z)
= 
\sum_{\alpha \in I^{(\bi(\xi))}}
\frac{a_{\alpha}}{\deg(g_{\alpha})}
\sum_{w\in g_{\alpha}^{-1}(z)}
\varphi(\alpha\xi,w), 
\]
where $(\xi, z) \in J(\tilde{f})$.
Since $\mathcal{L}_{a}1=1$, the dual operator $\mathcal{L}_{a}^{*}$ sends $\mathcal M(J(\tilde{f}))$ into itself.
The following proposition follows similarly to Proposition \ref{qQrOdN2Nmt}. 
\begin{prp}\label{6QlyKESWkv}
Let $a \in \mathcal{W}^{+}(S)$.
Then, there exists a probability measure $\tilde{\mu}_{a}\in \mathcal{M}(J(\tilde{f}))$ such that
$\mathcal{L}_{a}^{*}\tilde{\mu}_{a}=\tilde{\mu}_{a}$.
Moreover, $\tilde{\mu}_{a}$ is $\tilde{f}$-invariant.
\end{prp}

We call $\tilde{\mu}_{a}$ the \textit{stationary measure} associated with $a$.
For $a \in \mathcal{W}^{+}(S)$, we define the measure
\[
\nu_{a}:=(\pi_{1})_*\tilde\mu_a 
\]
on $X(S)$. 
Note that the measure $\nu_{a}$ is $\sigma$-invariant.
For each vertex $i \in V$, we define
\[
c_{i} := \nu_{a}(X_{i}(S))
\]
and have  $\sum_{i \in V}c_{i}=1$.
Moreover, one readily sees that for each $(\alpha_{1},\ldots,\alpha_n) \in X^{n}(S)$ that  $\nu_{a}([\alpha_{1},\ldots,\alpha_n])
= c_{\bt(\alpha_n)} \prod_{k=1}^n a_{\alpha_k}$ by using Proposition \ref{6QlyKESWkv}.

If $\mathfrak{P}_{1}$ and $\mathfrak{P}_{2}$ are two partitions of $J(\tilde{f})$, we define their join, $\mathfrak{P}_{1} \vee \mathfrak{P}_{2} :=\{A \cap B: A \in \mathfrak{P}_{1}, B \in \mathfrak{P}_{2}\}$ 
and  $\mathfrak{P}^{-}:=\bigvee_{n=1}^{\infty} \tilde{f}^{-n}(\mathfrak{P})$.

\begin{rmk}\label{l0jJwi7MjQ}
Note that 
$\#\tilde{f}^{-1}(\tilde{y}) = \sum_{\alpha \in I^{(\bi(\xi))}} \deg(g_{\alpha})
<\infty$ for every $\tilde{y}=(\xi, z) \in J(\tilde{f})$.
Moreover, as the restriction $\tilde{f}|_{J(\tilde{f})}$ is an open distance-expanding map, one can choose a finite Borel partition of sufficiently small diameter which is a one-sided generator of finite entropy on each ergodic component; see \cite[Lemma 3.5.5, Theorem 4.1.1]{MR2656475}. 
\end{rmk}

\begin{dfn}
Let $\mu$ be a Borel probability measure on $J(\tilde{f})$. 
Let $\tilde{\varepsilon}$ denote the point partition of $J(\tilde{f})$.
Let $(\mu_{\tilde y})_{\tilde y\in J(\tilde{f})}$ be probability measures with respect to $\tilde{f}^{-1}(\tilde\varepsilon)$ 
such that  $\mu_{\tilde y}$ is supported on $\tilde{f}^{-1}(\{\tilde y\})$ for $\mu$-a.e.  $\tilde y$ and
\[
\mu(A)=\int_{J(\tilde{f})} \mu_{\tilde y}(A) \mathrm{d}\mu(\tilde y)
\]
for a Borel set  $A\subset J(\tilde{f})$; see \cite[Theorems 2.6.7 and 2.6.11]{MR2656475}.
We define the \textit{conditional information function} of a partition $\tilde{\varepsilon}$ of $J(\tilde{f})$ given $\tilde{f}^{-1}(\tilde{\varepsilon})$ by
$I_{\mu}(\tilde{\varepsilon}\mid \tilde{f}^{-1}(\tilde{\varepsilon}))(\tilde{z}) := -\log \mu_{\tilde{f}(\tilde{z})}(\{\tilde{z}\})$ for $\mu$-a.e. $\tilde{z} \in J(\tilde{f})$, 
and the \textit{conditional entropy} by
$H_\mu(\tilde{\varepsilon}\mid \tilde{f}^{-1}(\tilde{\varepsilon})) := \int_{J(\tilde{f})} I_{\mu}(\tilde{\varepsilon}\mid \tilde{f}^{-1}(\tilde{\varepsilon}))(\tilde{z}) \mathrm{d}\mu(\tilde{z})$; see \cite[Definition 2.8.3]{MR2656475}.
\end{dfn}

\begin{lem}\label{TbxwSFQbox2}
Let $a := (a_{\alpha})_{\alpha \in I} \in \mathcal{W}^{+}(S)$.
Then, for $\tilde\mu_a$-a.e.  $\tilde{z}\in J(\tilde{f})$ and every $\tilde y=(\eta,w)\in \tilde{f}^{-1}(\tilde{z})$,
\[
\tilde\mu_{a,\tilde{z}}(\{\tilde y\})
=
\frac{a_{\eta_{1}}}{\deg(g_{\eta_{1}})},
\]
where $(\tilde\mu_{a,\tilde{z}})_{\tilde{z}}$ denotes the Rokhlin conditional measures
with respect to the partition $\tilde{f}^{-1}(\tilde\varepsilon)$.
\end{lem}

\begin{proof}
Let $\tilde{z}=(\xi,z)\in J(\tilde{f})$.
Define a probability measure $\ell_{\tilde{z}}$ on $J(\tilde{f})$ supported on $\tilde{f}^{-1}(\tilde{z})$ by
\[
\ell_{\tilde{z}}(B)
:=
\sum_{\tilde y=(\eta,w)\in \tilde{f}^{-1}(\tilde{z})}
\frac{a_{\eta_{1}}}{\deg(g_{\eta_{1}})} \mathbf 1_{B}(\tilde y)
\]
for each Borel set $B \subset J(\tilde{f})$. 
By the definition of $\mathcal{L}_{a}$, we have $(\mathcal{L}_{a}\mathbf 1_B)(\tilde{z})=\ell_{\tilde{z}}(B)$.
Since $\mathcal{L}_{a}^*\tilde\mu_a=\tilde\mu_a$, it follows that
\[
\tilde\mu_a(B)=\int (\mathcal{L}_{a}\mathbf 1_B)(\tilde{z}) d\tilde\mu_a(\tilde{z})
=\int \ell_{\tilde{z}}(B) d\tilde\mu_a(\tilde{z}).
\]
By the uniqueness of Rokhlin disintegration for $\tilde{f}^{-1}(\tilde\varepsilon)$ \cite[Theorems 2.6.7]{MR2656475},
we conclude that $\tilde\mu_{a,\tilde{z}}=\ell_{\tilde{z}}$ for $\tilde\mu_a$-a.e. $\tilde{z}$.
The proof is complete.
\end{proof}

\begin{lem}\label{XmrtH1fPmC}
Let $\tilde{\varepsilon}$ denote the point partition of $J(\tilde{f})$ and let  $a := (a_{\alpha})_{\alpha \in I} \in \mathcal{W}^{+}(S)$. 
Then, we have 
\begin{equation}\label{fDE62Sgc25}
H_{\tilde{\mu}_{a}}(\tilde{\varepsilon}\mid \tilde{f}^{-1}(\tilde{\varepsilon}))
=
\sum_{i \in V} c_{i}
\left(
-\sum_{\alpha \in I^{(i)}}a_{\alpha}\log a_{\alpha}
+
\sum_{\alpha \in I^{(i)}}a_{\alpha}\log\deg(g_{\alpha})
\right).
\end{equation}
\end{lem}

\begin{proof}
By Lemma \ref{TbxwSFQbox2}, for $\tilde{\mu}_{a}$-a.e.  $\tilde{z}=(\xi,z)\in J(\tilde{f})$ we have
\[
\tilde{\mu}_{a, \tilde{f}(\tilde{z})}(\{\tilde{z}\})
=\frac{a_{\xi_{1}}}{\deg(g_{\xi_{1}})}.
\]
Hence, for $\mu$-a.e.  $\tilde{z}=(\xi, z) \in J(\tilde{f})$,
\[
I_{\mu_{a}}(\tilde{\varepsilon}\mid \tilde{f}^{-1}(\tilde{\varepsilon}))(\xi, z)
=
-\log\mu_{\tilde{f}(\tilde{z})}(\{\tilde{z}\})
=
-\log a_{\xi_{1}}+\log\deg(g_{\xi_{1}}).
\]
and thus
\[
H_{\mu_{a}}(\tilde{\varepsilon}\mid \tilde{f}^{-1}(\tilde{\varepsilon}))
=
\int_{X(S)}\left(-\log a_{\xi_{1}} + \log\deg(g_{\xi_{1}})\right) \mathrm{d}\nu_{a}(\xi).
\]
Next, by $\sigma$-invariance of $\nu_{a}$ we have
\[
\nu_{a}(\{\bt(\xi_{1})=i\})
=\nu_{a}(\{\bi(\xi_{2})=i\})
=\nu_{a}(\sigma^{-1}(X_{i}(S)))
=\nu_{a}(X_{i}(S))=c_{i}.
\]
for each $i \in V$. 
Moreover, for each $i \in V$, the conditional distribution of $\xi_{1}$ given $\bt(\xi_{1})=i$ is the probability vector $(a_{\alpha})_{\alpha \in I^{(i)}}$.
Hence,
\[
\int_{\{\bt(\xi_{1})=i\}}\left(-\log a_{\xi_{1}}+\log\deg(g_{\xi_{1}})\right) \mathrm{d}\nu_{a}(\xi)
=
c_{i}\sum_{\alpha \in I^{(i)}} a_{\alpha}\left(-\log a_{\alpha}+\log\deg(g_{\alpha})\right).
\]
Summing over $i \in V$ gives \eqref{fDE62Sgc25}, which completes the proof. 
\end{proof}

\begin{dfn}\label{r8BiPogUOU}
Define the locally constant potential $\psi:X(S)\to\R$ by
$\psi(\xi):=\log\deg(g_{\xi_{1}})$ for $\xi \in X(S)$.
We also put $\tilde\psi:=\psi\circ\pi_{1}$.
Then, for $\tilde{z}=(\xi, z) \in J(\tilde{f})$ we have
$\calS_{n}\tilde\psi(\tilde{z})=\sum_{k=0}^{n-1}\psi(\sigma^{k}\xi)=\log\deg(g_{\xi|n})$.
For $n \in \N$ and $t\in \R$, we put
\begin{equation}\label{8ONjjP0FAw}
Z_n^{\deg}(t):=\sum_{\xi \in X^{n}(S)}\deg(g_{\xi})^t
\end{equation}
For $t\in \R$, we write
\[
\Press (\sigma,t\psi) = \lim_{n \to \infty}\frac{1}{n}\log Z_n^{\deg}(t).
\]
In particular, recalling $N_n:=\sum_{\xi \in X^{n}(S)}\deg(g_{\xi})=Z_n^{\deg}(1)$, we have
\[
\Press (\sigma, \psi) = \lim_{n \to \infty}\frac{1}{n}\log N_n.
\]
\end{dfn}

Let $\tilde{\varepsilon}$ be the point partition of $J(\tilde{f})$,
$\varepsilon_{X(S)}$ the point partition of $X(S)$, and
$\varepsilon_{\rsp}$ the point partition of $\rsp$.
Let $h_{m}(\tilde{f})$ denote the \textit{measure-theoretic entropy} of $\tilde{f}$ with respect to $m \in \mathcal{M}^{\tilde{f}}(J(\tilde{f}))$; see e.g. \cite{MR648108}.

\begin{lem}\label{t0nQ0VqNK1}
Let $\mu\in \mathcal{M}^{\tilde{f}}(J(\tilde{f}))$ and put
$\nu:=(\pi_{1})_*\mu\in \mathcal{M}^{\sigma}(X(S))$.
Then, we have
\begin{equation*}
h_\mu(\tilde{f})\leq h_\nu(\sigma)+\int_{X(S)}\psi \mathrm{d}\nu.
\end{equation*}
Consequently, we obtain $h_{\mathrm{top}}(\tilde{f}) \leq \log\rho(M)$.
\end{lem}

\begin{proof}
Fix $\mu\in \mathcal{M}^{\tilde{f}}(J(\tilde{f}))$ and set $\nu=(\pi_{1})_*\mu$.
Note that $\pi_{1}\circ \tilde{f}=\sigma\circ \pi_{1}$.
By \cite[Lemma 3.1]{MR476995}, we have $h_\mu(\tilde{f})=h_\nu(\sigma) + h_\mu(\tilde{f}\mid \sigma)$, where  $h_\mu(\tilde{f}\mid\sigma)$ denotes the relative entropy of $\tilde{f}$ with respect to  $\sigma$.
We show the inequality 
\begin{equation*}
h_\mu(\tilde{f}\mid \sigma)\leq \int_{X(S)}\psi \mathrm{d}\nu.
\end{equation*}
By the same argument of the proof of \cite[Theorem 6.10]{MR1767945}, we obtain 
\[
h_\mu(\tilde{f}\mid \sigma)
\leq 
H_\mu\left(\pi_{2}^{-1}(\varepsilon_{\rsp})\ \Big|\ 
\tilde{f}^{-1}(\tilde{\varepsilon})\ \vee\ \pi_{1}^{-1}(\varepsilon_{X(S)})\right).
\]
Let $\tilde{z}=(\xi, z) \in J(\tilde{f})$.
The partition element of $\tilde{f}^{-1}(\tilde{\varepsilon})\vee \pi_{1}^{-1}(\varepsilon_{X(S)})$ containing $\tilde{z}$ is
\[
E(\tilde{z})
:=
\left\{(\xi,w) \in J(\tilde{f}): g_{\xi|1}(w)=g_{\xi|1}(z)\right\}. 
\]
By \cite[Theorem 2.6.7]{MR2656475}, there exists a canonical system of conditional measures
\[
(\mu_E)_{E \in \tilde{f}^{-1}(\tilde\varepsilon)\vee \pi_{1}^{-1}(\varepsilon_{X(S)})}. 
\]
For $\mu$-a.e.  $\tilde{z}\in J(\tilde{f})$, we write $\mu_{E(\tilde{z})}$ for the conditional measure supported on the partition element $E(\tilde{z})$.
We also write $\pi_{2}^{-1}(\varepsilon_{\rsp})|E(\tilde{z}):=\{A \cap E(\tilde{z}) \colon  A \in \pi_{2}^{-1}(\varepsilon_{\rsp}), A \cap E(\tilde{z}) \neq \varnothing\}$. 
Since $\pi_{2}^{-1}(\varepsilon_{\rsp})$ induces on $E(\tilde{z})$ with at most
$\deg(g_{\xi|1})$ partition elements, we have
\[
\sum_{A \in \pi_{2}^{-1}(\varepsilon_{\rsp})|E(\tilde{z})} - \mu_{E(\tilde{z})}(A)\log\mu_{E(\tilde{z})}(A)
 \leq \log\deg(g_{\xi|1})
\ =\ \tilde\psi(\tilde{z})
\]
for $\mu$-a.e.  $\tilde{z} \in J(\tilde{f})$.
Integrating this estimate and using the definition of conditional entropy yields
\[
H_\mu \left(\pi_{2}^{-1}(\varepsilon_{\rsp})\ \Big|\ 
\tilde{f}^{-1}(\tilde{\varepsilon})\ \vee\ \pi_{1}^{-1}(\varepsilon_{X(S)})\right)
 \leq 
\int_{J(\tilde{f})} \tilde\psi \mathrm{d}\mu
=
\int_{X(S)}\psi \mathrm{d}\nu.
\]
Hence, the first assertion follows.
Finally, by the variational principle \cite[Theorem 9.10]{MR648108},
\[
h_{\mathrm{top}}(\tilde{f})
\le
\sup_{m\in \mathcal{M}^{\sigma}(X(S))}
\left(h_m(\sigma)+\int_{X(S)}\psi \mathrm{d}m\right)
=
\Press (\sigma,\psi).
\]
Using Definition \ref{r8BiPogUOU} and a similar argument to the proof of \cite[Theorem 7.13(ii)]{MR648108}, we obtain $\Press (\sigma,\psi)=\log\rho(M)$.
Therefore, we have $h_{\mathrm{top}}(\tilde{f})\le\log\rho(M)$.
\end{proof}

\begin{thm}\label{NBnbDkLhPN}
Let $t \in \R$ and $a^{(t)} := (a^{(t)}_{\alpha})_{\alpha \in I} \in \mathcal{W}^{+}(S)$.
Then, we have 
\[
h_{\tilde{\mu}_{a^{(t)}}}(\tilde{f})
+(t-1) \int_{J(\tilde{f})}
 \log\deg(g_{\xi_{1}}) \mathrm{d}\tilde{\mu}_{a^{(t)}}(\xi, z)
=
\log\rho\left(M^{(t)}\right).
\]
In particular,  we have $\Press (\sigma, \psi) = h_{\mathrm{top}}(\tilde{f}) =\log\rho(M)$.
\end{thm}

\begin{proof}
Let $\mu := \tilde\mu_{a^{(t)}}$ and $a:= a^{(t)}$.
Noting Remark \ref{l0jJwi7MjQ},  we have 
$h_{\mu}(\tilde{f}) = \int \log J_{\mu}(\tilde{f})(\tilde{z})  \mathrm{d}\mu(\tilde{z})$, 
where $J_{\mu}(\tilde{f})$ is the weak Jacobian of $\tilde{f}$ with respect to $\mu$; see \cite[Theorem 2.9.7]{MR2656475}.
By Lemma \ref{TbxwSFQbox2}, for $\mu$-a.e.  $\tilde{z}=(\xi,z)\in J(\tilde{f})$ we obtain
\[
\log J_{\mu}(\tilde{f})(\tilde{z})
=
\log\rho\left(M^{(t)}\right) + \log u^{(t)}_{\bt(\xi_{1})} - \log u^{(t)}_{\bi(\xi_{1})} - (t-1)\log\deg(g_{\xi_{1}}).
\]
Integrating this identity with respect to $\mu$ yields
\begin{align*}
h_{\mu}(\tilde{f})
=
\log\rho\left(M^{(t)}\right)
& + \int_{J(\tilde{f})} \left(\log u^{(t)}_{\bt(\xi_{1})} -\log u^{(t)}_{\bi(\xi_{1})}\right) \mathrm{d}\mu(\xi,z) \\
& -(t-1) \int_{J(\tilde{f})}\log\deg(g_{\xi_{1}}) \mathrm{d}\mu(\xi,z).
\end{align*}
It remains to show that the middle term is equal to $0$.
Let $\varphi := \varphi^{(t)} \colon X(S)\to\R$ be given by $\varphi(\zeta):=\log u^{(t)}_{\bi(\zeta)}$.
Since $\bt(\xi_{1})=\bi(\sigma(\xi))$ for every $\xi \in X(S)$, we have
$\log u^{(t)}_{\bt(\xi_{1})}=\varphi(\sigma(\xi))$ and $\log u^{(t)}_{\bi(\xi_{1})}=\varphi(\xi)$.
Therefore, we have 
\[
\int_{J(\tilde{f})} \log u^{(t)}_{\bt(\xi_{1})} \mathrm{d}\mu(\xi,z)
=
\int_{X(S)} \varphi(\sigma(\xi)) \mathrm{d}\nu_{a}(\xi) 
\]
and
\[
\int_{J(\tilde{f})} \log u^{(t)}_{\bi(\xi_{1})} \mathrm{d}\mu(\xi,z)
=
\int_{X(S)} \varphi(\xi) \mathrm{d}\nu_{a}(\xi).
\]
Since $\nu_{a}=(\pi_{1})_*\mu$ is $\sigma$-invariant, the right-hand sides coincide, hence the middle term is equal to $0$.
Therefore, we obtain
\[
h_{\mu}(\tilde{f})
+(t-1)
\int_{J(\tilde{f})}\log\deg(g_{\xi_{1}})\mathrm{d}\mu(\xi, z)
=
\log\rho\left(M^{(t)}\right).
\]
Taking $t=1$ yields
$h_{\tilde{\mu}_{a^{(1)}}}(\tilde{f}) = \log\rho(M^{(1)}) = \log\rho(M)$.
Since $h_{\mathrm{top}}(\tilde{f}) \geq h_{\tilde{\mu}_{a^{(1)}}}(\tilde{f}) =\log\rho(M)$ by the variational principle \cite[Theorem 8.6]{MR648108}, it follows from Lemma \ref{t0nQ0VqNK1} that $h_{\mathrm{top}}(\tilde{f}) = \log\rho(M)$, which completes the proof. 
\end{proof}

\section{Sequential capacity topological pressure}\label{GE4VqBJMCw}

In this section, we characterise the sequential capacity topological pressure for rational GDMSs in our setting.
\begin{SA*}
Throughout this section, we fix a finitely generated, irreducible rational GDMS $S=(V, E, (\Gamma_{e})_{e \in E})$.
We also fix $R>0$ and a family of vertex-wise holes $(\sfH_{i}(R))_{i \in V}$.
For each $j\in V$, we fix $\kappa_{j}\in X_j(S)$ such that $y_j\in J_{\kappa_{j}}$; see  Notation \ref{u58n7Yc3qy}.
Moreover, we assume that the associated skew product
$\tilde{f} \colon X(S)\times\rsp\circlearrowleft$ is expanding along fibres.
\end{SA*}

\begin{dfn}\label{FAYWI6yRa4}
For an open set $\Omega \subset \rsp$ and $\xi \in X^{*}(S)$, let $\Comp(\Omega,g_{\xi})$
denote the set of connected components of $g_{\xi}^{-1}(\Omega)$.
For $\xi \in X^{*}(S)$, put
\[
\Comp(\xi, R):=\Comp(D(y_{\bt(\xi)}, R), g_{\xi}).
\]
For each $U\in \Comp(\xi, R)$, let $U^{\prime}$ be the unique connected component of $g_{\xi}^{-1} (D(y_{\bt(\xi)},2R))$ containing $U$, and let
\[
\gamma_{\xi, U} \colon D(y_{\bt(\xi)},2R)\to U^{\prime}
\]
be the inverse branch of $g_{\xi}$ such that $g_{\xi}\circ\gamma_{\xi, U}(z)=z$ for all
$z \in D(y_{\bt(\xi)},2R)$.
For $\xi \in X^{*}(S)$ and $U\in \Comp(\xi, R)$, we write
\[
B_{\xi, U}(R):=
\gamma_{\xi, U}(\sfH_{\bt(\xi)}(R)) \subset J_{\bi(\xi)}(S), 
\quad \tilde{B}_{\xi, U}(R)
:=
\left([\xi]\times \gamma_{\xi, U}(\sfH_{\bt(\xi)}(R))\right)\cap J(\tilde{f}) 
\]
and
\[
\mathrm{w}_{\delta}(\xi, U)
:=
|(\gamma_{\xi, U})^{\prime}(y_{\bt(\xi)})|^{\delta}
=
\left|(g_{\xi})^{\prime}(\gamma_{\xi, U}(y_{\bt(\xi)}))\right|^{-\delta}.
\]
For $n \in \N$, we finally define
\[
Z_n^{\mathrm{geom}}(\delta, R)
:=
\sum_{\xi \in X^{n}(S)} \sum_{U\in \Comp(\xi, R)}
\mathrm{w}_{\delta}(\xi, U) 
\]
and
\[
\overline{\Press ^{\mathrm{geom}}}(\delta, R)
:=
\limsup_{n\to\infty}\frac1n\log Z_n^{\mathrm{geom}}(\delta, R)
\quad  \text{and} \quad 
\underline{\Press ^{\mathrm{geom}}}(\delta, R)
:=
\liminf_{n\to\infty}\frac1n\log Z_n^{\mathrm{geom}}(\delta, R).
\]
\end{dfn}

\begin{rmk}
Note that from Proposition \ref{HSecY02Qdj} it follows that for $\xi \in X^{*}(S)$ and $U\in \Comp(\xi, R)$ we have $\gamma_{\xi, U}(y_{\bt(\xi)}) \in  g_{\xi}^{-1}(J_{\kappa_{\bt(\xi)}})=J_{\xi\kappa_{\bt(\xi)}}$ and thus $\left(\xi\kappa_{\bt(\xi)}, \gamma_{\xi, U}(y_{\bt(\xi)})\right) \in J(\tilde{f})$. 

\end{rmk}

\begin{notation}\label{u58n7Yc3qy}
For $\xi \in X^{*}(S)$ and $U\in \Comp(\xi, R)$, we set
\[
\tilde{z}_{\xi, U}:=\left(\xi\kappa_{\bt(\xi)}, \gamma_{\xi, U}(y_{\bt(\xi)})\right) \in J(\tilde{f})
\]
and
\[
\mathcal{R}_{n}:=\{\tilde{z}_{\xi, U}:\ \xi \in X^{n}(S),\ U\in \Comp(\xi, R)\}.
\]
\end{notation}

\begin{lem}
For $n \in \N$, $\xi \in X^{n}(S)$ and  $U \in \Comp(\xi, R)$, we have $ \tilde{B}_{\xi, U}(R) \subset \widetilde{\sfHP}_{n}(R)$.
\end{lem}

\begin{proof}
Let $(\eta, z) \in \tilde{B}_{\xi, U}(R)$.
Then, we have $\eta \in [\xi]$ and there exists $\tau \in X_{\bt(\xi)}(S)$ such that $\eta=\xi\tau$.
Moreover, since $z\in \gamma_{\xi, U}(\sfH_{\bt(\xi)}(R))$, there exists  $u\in \sfH_{\bt(\xi)}(R)$ such that $g_\xi(z)=u$.
Therefore,
\[
\tilde{f}^{ n}(\eta, z)=(\sigma^n(\eta),g_\xi(z))=(\tau,u) \in X_{\bt(\xi)}(S)\times \sfH_{\bt(\xi)}(R).
\]
Since $(\eta, z) \in J(\tilde{f})$,  \eqref{l0mSj4Qp58} implies that $\tilde{f}^{ n}(\eta, z) \in J(\tilde{f})$.
Hence $\tilde{f}^{ n}(\eta, z) \in \widetilde{\sfH}(R)$, which proves $\tilde{B}_{\xi, U}(R)\subset \widetilde{\sfHP}_{n}(R)$ and completes the proof.
\end{proof}

The following lemma was observed in the proof of \cite[Proposition 4.4.3]{MR2656475}.
\begin{lem}\label{GSCyI6vVWN}
There exists $\eta_{0}>0$ such that for any $\tilde{y} \in J(\tilde{f})$ and  $n \in \N$ the set $\tilde{f}^{-n}(\tilde{y})$ is $(n,2\eta_{0})$-separated, i.e. for any distinct points $\tilde{z}, \tilde{z}^{\prime} \in \tilde{f}^{-n}(\tilde{y})$, we have $d_{n}(\tilde{z},\tilde{z}^{\prime})\geq 2\eta_{0}$.
\end{lem}

\begin{lem}\label{iQ3l1o1cKC}
There exists $\varepsilon_{\mathrm{rep}}>0$ such that for every $n \in \N$, the finite set $\mathcal{R}_{n}$ is $(n,2\varepsilon_{\mathrm{rep}})$-separated in $J(\tilde{f})$.
Further, every Bowen ball $B_{n}(\tilde{z},\varepsilon_{\mathrm{rep}})$ intersects $\mathcal{R}_{n}$
in at most one point.
\end{lem}

\begin{proof}
Let $n \in \N$. 
Let $\eta_{0}>0$ be as in Lemma \ref{GSCyI6vVWN} and set $\varepsilon_{\mathrm{rep}}:=\min\{1/4,\eta_{0}/2\}$.
Take $\tilde{z}_{\xi, U}\neq \tilde{z}_{\xi^{\prime}, U^{\prime}}$, where $\xi := (\xi_{1}, \xi_{2}, \ldots, \xi_{n}), \xi^{\prime} := (\xi^{\prime}_{1}, \xi^{\prime}_{2}, \ldots, \xi^{\prime}_{n}) \in X^{n}(S)$.
Suppose that $\xi\neq\xi^{\prime}$.
Let $m := \min\{k \leq n \colon \xi_{k} \neq \xi^{\prime}_{k}\}$.
Then, we have 
\[
d_{X(S)} \left(\sigma^{m-1}(\xi\kappa_{\bt(\xi)}),\ \sigma^{m-1}(\xi^{\prime}\kappa_{\bt(\xi^{\prime})})\right)=\frac12.
\]
Therefore
$d_{n}(\tilde{z}_{\xi, U},\tilde{z}_{\xi^{\prime}, U^{\prime}})\geq 1/2\geq 2\varepsilon_{\mathrm{rep}}$.

Next, we suppose that $\xi=\xi^{\prime}$.
Then, $\tilde{f}^{n}(\tilde{z}_{\xi, U})=(\kappa_{\bt(\xi)},y_{\bt(\xi)})
=\tilde{f}^{n}(\tilde{z}_{\xi, U^{\prime}})$,
so both lie in $\tilde{f}^{-n}((\kappa_{\bt(\xi)}, y_{\bt(\xi)}))\cap J(\tilde{f})$.
By Lemma \ref{GSCyI6vVWN}, this set is $(n,2\eta_{0})$-separated, hence
\[
d_{n}(\tilde{z}_{\xi, U},\tilde{z}_{\xi, U^{\prime}})\geq 2\eta_{0}\geq 2\varepsilon_{\mathrm{rep}}.
\]
Finally, let $\tilde{z} \in J(\tilde{f})$.
If $B_{n}(\tilde{z},\varepsilon_{\mathrm{rep}})$ intersected $\mathcal{R}_{n}$ in two distinct points $\tilde{z}_{1},\tilde{z}_{2}\in \mathcal{R}_{n}$, then for $d_n$ we would have
\[
d_n(\tilde{z}_{1},\tilde{z}_{2})
\leq d_n(\tilde{z}_{1},\tilde{z})+d_n(\tilde{z},\tilde{z}_{2})
< \varepsilon_{\mathrm{rep}}+\varepsilon_{\mathrm{rep}}
=2\varepsilon_{\mathrm{rep}},
\]
contradicting the $(n,2\varepsilon_{\mathrm{rep}})$-separation of $\mathcal{R}_{n}$.
Hence, every $B_{n}(\tilde{z},\varepsilon_{\mathrm{rep}})$ intersects $\mathcal{R}_{n}$
in at most one point.
The proof is complete. 
\end{proof}

\begin{lem}\label{7xObqtHg0Y}
There exists a constant $K_{\mathrm{koe}}\ge1$ such that for any
$n \in \N$, $\xi \in X^{n}(S)$, $U \in \Comp(\xi, R)$ and 
$\tilde{z} \in \tilde{B}_{\xi, U}(R)$ we have 
\[
K_{\mathrm{koe}}^{-1} \mathrm{w}_{\delta}(\xi, U)
 \leq 
\exp(\calS_{n}(\delta\tilde{\varphi})(\tilde{z}))
 \leq 
K_{\mathrm{koe}} \mathrm{w}_{\delta}(\xi, U).
\]
\end{lem}

\begin{proof}
Fix $\xi \in X^{n}(S)$ and $U\in \Comp(\xi, R)$.
Take an arbitrary point $\tilde{z} \in \tilde{B}_{\xi, U}(R)$ such that $\pi_{1}(\tilde{z}) \in [\xi]$.
Denote $z := \pi_{2}(\tilde{z})$. 
By definition of $\tilde{B}_{\xi, U}(R)$, there exists a point
$u\in \sfH_{\bt(\xi)}(R)\subset D(y_{\bt(\xi)}, R)$ such that
$z=\gamma_{\xi, U}(u)$.
Hence, we have $\exp(\calS_{n}(\delta\tilde\varphi)(\tilde{z}))
=|(g_{\xi})^{\prime}(z)|^{-\delta}
=|(\gamma_{\xi, U})^{\prime}(u)|^{\delta}$.
%t
Applying Proposition \ref{ekoebe1} to $D(y_{\bt(\xi)}, R) \subset D(y_{\bt(\xi)},2R)$ yields
\[
\sup_{w\in D(y_{\bt(\xi)}, R)}|(\gamma_{\xi, U})^{\prime}(w)|
\leq \bsk(1/2) 
\inf_{w\in D(y_{\bt(\xi)}, R)}|(\gamma_{\xi, U})^{\prime}(w)|,
\]
which implies that for each $u \in D(y_{\bt(\xi)}, R)$, we have
\[
\bsk(1/2)^{-1}|(\gamma_{\xi, U})^{\prime}(y_{\bt(\xi)})|
\leq |(\gamma_{\xi, U})^{\prime}(u)|
\leq \bsk(1/2)|(\gamma_{\xi, U})^{\prime}(y_{\bt(\xi)})|.
\]
Raising to the power $\delta>0$ yields
\[
\bsk(1/2)^{-\delta} |(\gamma_{\xi, U})^{\prime}(y_{\bt(\xi)})|^{\delta}
\leq |(\gamma_{\xi, U})^{\prime}(u)|^{\delta}
\leq \bsk(1/2)^{\delta} |(\gamma_{\xi, U})^{\prime}(y_{\bt(\xi)})|^{\delta}.
\]
Finally, since $\mathrm{w}_{\delta}(\xi, U)=|(\gamma_{\xi, U})^{\prime}(y_{\bt(\xi)})|^{\delta}$,
the desired estimate follows with $K_{\mathrm{koe}}:=\bsk(1/2)^{\delta}$.
\end{proof}

The following lemma follows from \cite[Lemma 4.4.2 (the Pre-Bounded Distortion Lemma)]{MR2656475}.
\begin{lem}\label{L3EQBJ1CFR}
There exist $\varepsilon_{\mathrm{bd}}>0$ and $D_{\mathrm{bd}}\geq 1$ such that for any $n \in \N$ and $\tilde{z} \in J(\tilde{f})$ we have 
\[
\exp\left(\sup_{\tilde{w} \in B_{n}(\tilde{z},\varepsilon_{\mathrm{bd}})} \calS_{n}(\delta\tilde{\varphi})(\tilde{w})\right)
\le
D_{\mathrm{bd}}\cdot \exp(\calS_{n}(\delta\tilde{\varphi})(\tilde{z})).
\]
\end{lem}

\begin{lem}\label{UjXdpkDB0w}
For every $0<\varepsilon\leq 1$, there exists $N_{\mathrm{cov}}(\varepsilon) \in \N$ with the following property:
for any $n \in \N$,  $\xi \in X^{n}(S)$ and  $U \in \Comp(\xi, R)$,
there exist points $\tilde{z}_{\xi, U,1},\dots,\tilde{z}_{\xi, U,N} \in \tilde{B}_{\xi, U}(R)$
for some $1\leq N\leq N_{\mathrm{cov}}(\varepsilon)$ such that
\[
\tilde{B}_{\xi, U}(R) =
\left([\xi]\times \gamma_{\xi, U}(\sfH_{\bt(\xi)}(R))\right)\cap J(\tilde{f})  \subset \bigcup_{m=1}^{N} B_{n}(\tilde{z}_{\xi, U,m},\varepsilon).
\]
\end{lem}

\begin{proof}
Fix $0<\varepsilon\leq 1$.
Choose $L=L(\varepsilon) \in \N$ such that $2^{-(L+1)}<\varepsilon/2$.
Since $S$ is finitely generated, we have 
\[
K(\varepsilon):=\max_{i \in V}\#X_{i}^{L}(S)<\infty.
\]

Fix $n \in \N$ and $\xi \in X^{n}(S)$.
Then, we have  $[\xi] \subset \bigcup_{\tau\in X_{\bt(\xi)}^{L}(S)} [\xi\tau]$. 
Moreover, for any $\eta,\eta^{\prime} \in [\xi\tau]$ and any $0\leq k<n$ we have
\[
d_{X(S)} \left(\sigma^{k}(\eta),\sigma^{k}(\eta^{\prime})\right)\leq 2^{-(L+1)}<\varepsilon/2.
\]
We now additionally fix $U\in \Comp(\xi, R)$.
For an integer $k \in \{0, \ldots, n\}$, we define a univalent function
\[
F_k:=g_{\xi|k}\circ \gamma_{\xi, U}:D(y_{\bt(\xi)},2R)\to \C,
\]
where $g_{\xi|0}:=\id$.
Let $h_{k}:=g_{\xi, n}\circ\cdots\circ g_{\xi, k + 1}$, where  $h_n:=\id$.

We show that each $F_k$ is Lipschitz on $D(y_{\bt(\xi)}, R)$. 
Let $\tilde{z}_{\xi, U}$ be as in Notation \ref{u58n7Yc3qy}. 
Then, by \eqref{l0mSj4Qp58}, we have $\tilde{f}^k(\tilde{z}_{\xi, U}) \in J(\tilde{f})$ for all $k \in \N_{0}$.
For each $k \in \{0, \ldots, n - 1\}$, we have
\[
|h_{k}^{\prime}(F_k(y_{\bt(\xi)}))|
=
\left|(\tilde{f}^{ n-k})^{\prime}(\tilde{f}^{k}(\tilde{z}_{\xi, U}))\right|
\geq c \lambda^{n-k},
\]
hence
\[
|F_{k}^{\prime}(y_{\bt(\xi)})|
=
\frac{1}{|h_{k}^{\prime}(F_k(y_{\bt(\xi)}))|}
\leq \frac{1}{c \lambda^{n-k}}
\leq \frac{1}{c \lambda}.
\]
For $k=n$, we have $F_n=\id$, so $|F_n^{\prime}(y_{\bt(\xi)})|=1$.
Therefore, for all $k \in \{0, \ldots, n\}$,
\[
|F_{k}^{\prime}(y_{\bt(\xi)})|\leq C_0:=\max\left\{1,\frac{1}{c\lambda}\right\}.
\]
By Proposition \ref{ekoebe1}, we have 
\[
\sup_{u\in D(y_{\bt(\xi)}, R)} |F_{k}^{\prime}(u)|
\leq \bsk(1/2) \inf_{u\in D(y_{\bt(\xi)}, R)}|F_{k}^{\prime}(u)|
\leq \bsk(1/2) |F_{k}^{\prime}(y_{\bt(\xi)})|
\leq \bsk(1/2) C_0
\]
for all $k \in \{0, \ldots, n\}$.
Hence each $F_k$ is $L_{\mathrm{Lip}}:=\bsk(1/2)C_{0}$-Lipschitz on $D(y_{\bt(\xi)}, R)$.

Choose $r>0$ such that $2L_{\mathrm{Lip}} r<\varepsilon/2$.
Since each $\sfH_{i}(R) \subset D(y_{i}, R)$ is totally bounded and $V$ is finite,
for each $j\in V$ there exist $u_{j,1},\dots,u_{j,M_j(\varepsilon)} \in \sfH_j(R)$ such that
\[
\sfH_j(R)\subset \bigcup_{\ell=1}^{M_j(\varepsilon)}D(u_{j,\ell}, r).
\]
Put $M(\varepsilon):=\max_{j\in V}M_j(\varepsilon)$.
For $\tau\in X_{\bt(\xi)}^L(S)$ and $\ell\in \{1,\dots,M_{\bt(\xi)}(\varepsilon)\}$, we define
\[
A_{\tau,\ell}:=
\left([\xi\tau]\times \gamma_{\xi, U}(\sfH_{\bt(\xi)}(R)\cap D(u_{\bt(\xi),\ell}, r))\right)\cap J(\tilde{f}).
\]
Then, we have $\tilde{B}_{\xi, U}(R) \subset \bigcup_{\tau\in X_{\bt(\xi)}^L(S)}\bigcup_{\ell\in \{1,\dots,M(\varepsilon)\}} A_{\tau,\ell}$.

If $A_{\tau,\ell}=\varnothing$, then $A_{\tau,\ell}\subset B_{n}(\tilde{z},\varepsilon)$ holds trivially for any $\tilde{z} \in J(\tilde{f})$.
Hence, we take an arbitrary point $\tilde{z}_{\tau,\ell} \in A_{\tau,\ell} \neq \varnothing$ and claim that $A_{\tau,\ell}\subset B_{n}(\tilde{z}_{\tau,\ell},\varepsilon)$.
Take $\tilde{z}=(\eta, z)$, $\tilde{z}^{\prime}=(\eta^{\prime}, z^{\prime}) \in A_{\tau,\ell}$.
Then $\eta,\eta^{\prime} \in [\xi\tau]$, hence for all $k \in \{0, \ldots, n - 1\}$,
\[
d_{X(S)}\left(\sigma^{k}(\eta),\sigma^{k}(\eta^{\prime})\right)<\varepsilon/2.
\]
Also, since $z, z^{\prime} \in \gamma_{\xi, U}(\sfH_{\bt(\xi)}(R))$, we have $u:=g_\xi(z) \in \sfH_{\bt(\xi)}(R)$, $v:=g_\xi(z^{\prime}) \in \sfH_{\bt(\xi)}(R)$.
Moreover $u,v\in D(u_{\bt(\xi),\ell},r)$, so $d_{\C}(u, v)<2r$.
For $k \in \{0, \ldots, n - 1\}$, using $\eta,\eta^{\prime} \in [\xi]$ we have
\[
\pi_{2}(\tilde{f}^{ k}(\tilde{z}))=g_{\xi|k}(z)=F_k(u),
\qquad
\pi_{2}(\tilde{f}^{ k}(\tilde{z}^{\prime}))=g_{\xi|k}(z^{\prime})=F_k(v),
\]
and thus
\[
d_{\C}\left(\pi_{2}(\tilde{f}^{ k}(\tilde{z})),\pi_{2}(\tilde{f}^{ k}(\tilde{z}^{\prime}))\right)
=
d_{\C}(F_k(u), F_k(v))
\leq L_{\mathrm{Lip}} d_{\C}(u, v)
< 2L_{\mathrm{Lip}} r
<\varepsilon/2.
\]
Hence for all $0\leq k<n$,
\[
d_{X(S)\times\C}\left(\tilde{f}^{ k}(\tilde{z}),\tilde{f}^{ k}(\tilde{z}^{\prime})\right)
=
\max\left\{
d_{X(S)}(\sigma^{k}(\eta),\sigma^{k}(\eta^{\prime})),
 d_{\C}\left(\pi_{2}(\tilde{f}^{ k}(\tilde{z})),\pi_{2}(\tilde{f}^{ k}(\tilde{z}^{\prime}))\right)
\right\}
<\varepsilon,
\]
so $d_n(\tilde{z},\tilde{z}^{\prime})<\varepsilon$, proving $A_{\tau,\ell}\subset B_{n}(\tilde{z}_{\tau,\ell},\varepsilon)$.

Finally, since we have 
$\#X_{\bt(\xi)}^{L}(S)\cdot M(\varepsilon)\leq K(\varepsilon) M(\varepsilon)$, the set $\tilde{B}_{\xi, U}(R)$ is covered by at most $N_{\mathrm{cov}}(\varepsilon):=K(\varepsilon) M(\varepsilon)$
Bowen balls of the form $B_{n}(\cdot,\varepsilon)$ with centers in $\tilde{B}_{\xi, U}(R)$. 
The proof is complete.
\end{proof}

\begin{prp}\label{vGhPS8pfoC}
Let $S=(V, E, (\Gamma_{e})_{e \in E})$ be a finitely generated, irreducible rational GDMS.
Assume that the associated skew product $\tilde{f}$ is expanding along fibres and topologically transitive on $J(\tilde{f})$.
Then, we have
\[
\overline{\calCP}\left(\tilde{f}|_{J(\tilde{f})}, \delta\tilde{\varphi},\widetilde{\sfHP}(R)\right)
=
\overline{\Press^{\mathrm{geom}}}(\delta, R).
\]
\end{prp}

\begin{proof}
Put $\varepsilon_0:=\min\{\varepsilon_{\mathrm{bd}},1,\varepsilon_{\mathrm{rep}}\}$.
Fix $n \in \N$.
For any $\xi \in X^{n}(S)$, $U \in \Comp(\xi, R)$, by Lemma \ref{UjXdpkDB0w} there exist $\tilde{z}_{\xi, U,1},\dots,\tilde{z}_{\xi, U,N} \in \tilde{B}_{\xi, U}(R)$ such that 
\[
\tilde{B}_{\xi, U}(R) \subset \bigcup_{m=1}^{N_{\mathrm{cov}}(\varepsilon_0)} B_{n}(\tilde{z}_{\xi, U,m},\varepsilon_0). 
\]
Noting the definition of $\widetilde{\sfHP}_{n}(R)$ given in \eqref{sqBBlOyDPn}, we also have 
\[
\widetilde{\sfHP}_{n}(R)
\subset
\bigcup_{\xi \in X^{n}(S)} \bigcup_{U \in \Comp(\xi, R)}  \tilde{B}_{\xi, U}(R)
\]
Indeed, for $(\omega, z) \in \widetilde{\sfHP}_{n}(R)$, we have $g_\omega(z) \in \sfH_{\bt(\omega)}(R)\subset D(y_{\bt(\omega)},R)$, hence $z$ belongs to a unique component $U\in \Comp(\omega,R)$,
and the inverse branch $\gamma_{\omega, U}$ is well-defined on $D(y_{\bt(\omega)},2R)$.
Hence, 
\[
\Lambda_n(\widetilde{\sfHP}_{n}(R),\delta\tilde{\varphi},\varepsilon_0)
\le
\sum_{\xi \in X^{n}(S)} \sum_{U \in \Comp(\xi, R)} \sum_{m=1}^{N_{\mathrm{cov}}(\varepsilon_0)}
\exp\left(\sup_{\tilde{y} \in B_{n}(\tilde{z}_{\xi, U,m},\varepsilon_0)} \calS_{n}(\delta\tilde{\varphi})(\tilde{y})\right).
\]
By Lemma \ref{L3EQBJ1CFR},
\[
\exp\left(\sup_{\tilde{y} \in B_{n}(\tilde{z}_{\xi, U,m},\varepsilon_0)} \calS_{n}(\delta\tilde{\varphi})(\tilde{y})\right)
\le
D_{\mathrm{bd}}\cdot \exp(\calS_{n}(\delta\tilde{\varphi})(\tilde{z}_{\xi, U,m}));
\]
since $\tilde{z}_{\xi, U,m} \in \tilde{B}_{\xi, U}(R)$, Lemma \ref{7xObqtHg0Y} yields
$\exp(\calS_{n}(\delta\tilde{\varphi})(\tilde{z}_{\xi, U,m}))\leq K_{\mathrm{koe}} \mathrm{w}_{\delta}(\xi, U)$.
Therefore, we obtain 
\begin{align}
\Lambda_n(\widetilde{\sfHP}_{n}(R),\delta\tilde{\varphi},\varepsilon_0)
&\notag \le
D_{\mathrm{bd}} K_{\mathrm{koe}} N_{\mathrm{cov}}(\varepsilon_0) 
\sum_{\xi \in X^{n}(S)} \sum_{U \in \Comp(\xi, R)}\mathrm{w}_{\delta}(\xi, U)\\
&\label{i26X1C6icp} =
D_{\mathrm{bd}} K_{\mathrm{koe}} N_{\mathrm{cov}}(\varepsilon_0) Z_n^{\mathrm{geom}}(\delta, R).
\end{align}

Next, take an arbitrary at most countable set
$E=\{\tilde{z}\}_{\tilde{z} \in E}\subset J(\tilde{f})$
such that
$\widetilde{\sfHP}_{n}(R) \subset \bigcup_{\tilde{z} \in E} B_{n}(\tilde{z},\varepsilon_0)$.
Since $\varepsilon_0\leq \varepsilon_{\mathrm{rep}}$, Lemma \ref{iQ3l1o1cKC} implies each ball
intersects $\mathcal{R}_{n}$ in at most one point. Hence,
\[
\sum_{\tilde{z} \in E} \exp\left(\sup_{\tilde{y} \in B_{n}(\tilde{z},\varepsilon_0)}
\calS_{n}(\delta\tilde{\varphi})(\tilde{y})\right)
\ge
\sum_{\tilde{z} \in \mathcal{R}_{n}} \exp(\calS_{n}(\delta\tilde{\varphi})(\tilde{z})).
\]
Observe that for $\xi \in X^{*}(S)$ and $U\in \Comp(\xi, R)$, we have 
\[
\exp(\calS_{n}(\delta\tilde{\varphi})(\tilde{z}_{\xi, U}))
=|(g_\xi)^{\prime}(\gamma_{\xi, U}(y_{\bt(\xi)}))|^{-\delta} 
=|(\gamma_{\xi, U})^{\prime}(y_{\bt(\xi)})|^{\delta} = \mathrm{w}_{\delta}(\xi, U)
\]
and thus
\[
\sum_{\tilde{z} \in \mathcal{R}_{n}} \exp(\calS_{n}(\delta\tilde{\varphi})(\tilde{z}))
=
\sum_{\xi \in X^{n}(S)} \sum_{U \in \Comp(\xi, R)}\mathrm{w}_{\delta}(\xi, U)
=
Z_n^{\mathrm{geom}}(\delta, R).
\]
Taking the infimum over all Bowen covers yields
\begin{equation}\label{RDVyNQstlu}
\Lambda_n(\widetilde{\sfHP}_{n}(R),\delta\tilde{\varphi},\varepsilon_0)\ \ge\ Z_n^{\mathrm{geom}}(\delta, R).
\end{equation}
Setting $C_{\mathrm{cap}}
:=
\max\left\{1,\ D_{\mathrm{bd}} K_{\mathrm{koe}} N_{\mathrm{cov}}(\varepsilon_0)\right\}$ and combining \eqref{i26X1C6icp} and \eqref{RDVyNQstlu}, 
we obtain
\[
C_{\mathrm{cap}}^{-1} Z_n^{\mathrm{geom}}(\delta, R)
\le
\Lambda_n(\widetilde{\sfHP}_{n}(R),\delta\tilde{\varphi},  \varepsilon_0)
\le
C_{\mathrm{cap}} Z_n^{\mathrm{geom}}(\delta, R), 
\]
which implies 
\[
\overline{\calCP}\left(\tilde{f},\delta\tilde{\varphi},\widetilde{\sfHP}(R), \varepsilon_0\right)
=
\overline{\Press^{\mathrm{geom}}}(\delta, R).
\]

For any $0<\varepsilon\le\varepsilon_0$, the same lower bound holds since $\varepsilon\leq \varepsilon_{\mathrm{rep}}$,
and the upper bound holds with $N_{\mathrm{cov}}(\varepsilon)$ in place of $N_{\mathrm{cov}}(\varepsilon_0)$.
Hence for all $0<\varepsilon\le\varepsilon_0$, we have 
$\overline{\calCP}\left(\tilde{f},\delta\tilde{\varphi},\widetilde{\sfHP}(R), \varepsilon\right)
=
\overline{\Press^{\mathrm{geom}}}(\delta, R)$.
Finally, Definition \ref{BphPCB0UAa} implies
$\overline{\calCP}\left(\tilde{f},\delta\tilde{\varphi},\widetilde{\sfHP}(R)\right)
=
\overline{\Press^{\mathrm{geom}}}(\delta, R)$.
\end{proof}

\section{Conformal preimage decay exponent}
Throughout this section, we assume the standing assumption given in the previous section.
Recall the definitions of $J_{i}(\tilde{f})$ in \eqref{h6uKK2bqKG} and the conformal measure given in Proposition \ref{tg3o75180I5}.
For each $i \in V$, we define the finite Borel measure $\nu_{\delta,i}$ on $J_{i}(S)$ given by 
$\nu_{\delta,i}:=(\pi_{2}|_{J_i(\tilde{f})})_*(\tilde\nu_{\delta}|_{J_i(\tilde{f})})$.

\begin{prp}[{\cite{arimitsu2024}}]\label{EwoyAkFqwq}
Assume that $S$ satisfies the VSC, and that the associated skew product map $\tilde{f} \colon X(S)\times\rsp\circlearrowleft$ is expanding along fibres and topologically transitive on $J(\tilde{f})$.
Then, there exist constants $C_{\mathrm{Ahl}}\geq 1$ and $r_0>0$ such that 
\begin{enumerate}[label=(\roman*)]
\item
for every $z\in J(S)$ and every $0<r<r_0$,
\[
C_{\mathrm{Ahl}}^{-1} r^{\delta}
 \leq 
\nu_{\delta} \left(D(z, r)\cap J(S)\right)
 \leq 
C_{\mathrm{Ahl}}  r^{\delta};
\]
\item 
for every $i \in V$, every $z\in J_{i}(S)$ and every $0<r<r_0$,
\[
C_{\mathrm{Ahl}}^{-1} r^{\delta}
 \leq 
\nu_{\delta,i} \left(D(z, r)\cap J_{i}(S)\right)
 \leq 
C_{\mathrm{Ahl}}  r^{\delta}.
\]
\end{enumerate}
\end{prp}

From now on, we additionally assume that $S$ satisfies the VSC and that
$\tilde{f}$ is topologically transitive on $J(\tilde{f})$.

\begin{lem}\label{kkvM33kTyT}
For $\xi \in X^{*}(S)$ and  $U\in \Comp(\xi, R)$, let $\gamma_{\xi, U}:D(y_{\bt(\xi)},2R)\to \C$ be as in Definition \ref{FAYWI6yRa4}.
Then, there exists a constant $D\geq 1$ such that
for any $\xi \in X^{*}(S)$ and  $U\in \Comp(\xi, R)$ we have
\[
D^{-1}|(\gamma_{\xi, U})^{\prime}(y_{\bt(\xi)})| d_{\C}(u, v)
 \leq 
d_{\C}(\gamma_{\xi, U}(u), \gamma_{\xi, U}(v))
 \leq 
D |(\gamma_{\xi, U})^{\prime}(y_{\bt(\xi)})| d_{\C}(u, v)
\]
for any $u,v\in D(y_{\bt(\xi)}, R)$.
\end{lem}

\begin{proof}
Fix $\xi \in X^{*}(S)$ and  $U\in \Comp(\xi, R)$.
Let
\[
F(z):=\frac{\gamma_{\xi, U}(y_{\bt(\xi)}+2Rz)-\gamma_{\xi, U}(y_{\bt(\xi)})}{2R (\gamma_{\xi, U})^{\prime}(y_{\bt(\xi)})}\qquad(|z|<1). 
\]
Then, $F$ belongs to the Schlicht class $\mathfrak{S}$.
Let $K:=\overline{D(0,1/2)}$ and define
\[
Q(F,u,v):=
\begin{cases}
\dfrac{d_{\C}(F(u), F(v))}{d_{\C}(u, v)}, & u\neq v,\\[4pt]
|F^{\prime}(u)|, & u=v.
\end{cases}
\]
Since the Schlicht class $\mathfrak{S}$ is compact in the topology of locally uniform convergence and $Q$ is continuous and positive
on $\mathfrak{S} \times K\times K$, there exist $0<m\leq M<\infty$ such that
$m\leq Q(F,u,v)\leq M$ for all $F\in \mathfrak{S}$ and $u,v\in K$.
Setting $D:=\max\{M,m^{-1}\}$ completes the proof.
\end{proof}

\begin{lem}\label{Ra4DgHuaXm}
There exist constants $C_{1},C_{2}>0$ such that
for any $\xi \in X^{*}(S)$ with $\bt(\xi)=j$ and any $U \in \Comp(\xi, R)$ we have 
\[
C_{1} \mathrm{w}_{\delta}(\xi, U) \leq 
\nu_{\delta}\left(\gamma_{\xi, U}(\sfH_{j}(R))\right)\ 
\le\ C_{2} \mathrm{w}_{\delta}(\xi, U).
\]
\end{lem}

\begin{proof}
Fix $\xi \in X^{*}(S)$ and $U\in \Comp(\xi, R)$.
Noting that $V$ is finite, by Proposition \ref{EwoyAkFqwq} and \cite[Theorem 6.9]{MR1333890}, we see that there exists a constant $C_{\rm H}\ge1$ such that
\begin{equation}\label{ksN5PsQKQC}
C_{\rm H}^{-1} \mathcal{H}^{\delta}(U) \leq \nu_{\delta}(U) \leq C_{\rm H} \mathcal{H}^{\delta}(U)
\end{equation}
for all Borel subset $U\subset J(S)$; and 
\begin{equation}\label{a3Kty4MVkf}
C_{\rm H}^{-1} \mathcal{H}^{\delta}(U) \leq \nu_{\delta, i}(U) \leq C_{\rm H} \mathcal{H}^{\delta}(U)
\end{equation}
for all Borel subset $U\subset J_{i}(S)$ and $i \in V$.

Let $r_0>0$ and $C_{\mathrm{Ahl}}\ge1$ be as in Proposition \ref{EwoyAkFqwq} and put
$r_{*}:=\min\{R,r_0/2\}>0$.
Then for every $i \in V$ we have
\[
\nu_{\delta, i} \left(\sfH_{i}(R)\right)\ \ge\
\nu_{\delta, i} \left(D(y_{i},r_{*})\cap J_{i}(S)\right)
\ \ge\ C_{\mathrm{Ahl}}^{-1}r_{*}^{\delta}.
\]
Combining this with \eqref{a3Kty4MVkf} and $\nu_{\delta,i}(\sfH_i(R))\le \nu_{\delta,i}(J_i(S))<\infty$ gives
\[
C_{\rm H}^{-1}C_{\mathrm{Ahl}}^{-1}r_{*}^{\delta} \leq \mathcal{H}^{\delta}(\sfH_{i}(R)) \leq C_{\rm H}
\]
for each $i \in V$. 
By Lemma \ref{kkvM33kTyT}, there exists $D\ge1$ such that
for all $u,v\in D (y_{\bt(\xi)}, R)$,
\[
D^{-1}|(\gamma_{\xi, U})^{\prime}(y_{\bt(\xi)})| d_{\C}(u, v)
 \leq 
d_{\C}(\gamma_{\xi, U}(u), \gamma_{\xi, U}(v))
 \leq 
D |(\gamma_{\xi, U})^{\prime}(y_{\bt(\xi)})| d_{\C}(u, v).
\]
Hence,  as $\gamma_{\xi, U}|_{D(y_{\bt(\xi)}, R)}$ and  $(\gamma_{\xi, U})^{-1}|_{\gamma_{\xi, U}(D(y_{\bt(\xi)}, R))}$ are both Lipschitz, we have 
\[
D^{-\delta}|(\gamma_{\xi, U})^{\prime}(y_{\bt(\xi)})|^{\delta} \mathcal{H}^{\delta}(\sfH_{\bt(\xi)}(R))
 \leq \mathcal{H}^{\delta}(\gamma_{\xi, U}(\sfH_{\bt(\xi)}(R))) \leq 
D^{\delta}|(\gamma_{\xi, U})^{\prime}(y_{\bt(\xi)})|^{\delta} \mathcal{H}^{\delta}(\sfH_{\bt(\xi)}(R)).
\]
Using \eqref{ksN5PsQKQC},  we obtain constants $C_{1},C_{2}>0$  such that
\[
C_{1} |(\gamma_{\xi, U})^{\prime}(y_{\bt(\xi)})|^{\delta}
 \leq \nu_{\delta}(\gamma_{\xi, U}(\sfH_{\bt(\xi)}(R))) \leq C_{2} |(\gamma_{\xi, U})^{\prime}(y_{\bt(\xi)})|^{\delta}.
\]
Finally, $\mathrm{w}_{\delta}(\xi, U)=|(\gamma_{\xi, U})^{\prime}(y_{\bt(\xi)})|^{\delta}$ by definition.
This completes the proof.
\end{proof}

\begin{prp}\label{grxaIp1jSL}
Let $\lambda>1$ be as in \eqref{3co5RIbc1p}.
Then
\[
\overline{\Press^{\mathrm{geom}}}(\delta, R) \leq \Press(\sigma, \psi)\ -\ \delta\log\lambda
\ <\ \Press(\sigma, \psi).
\]
\end{prp}

\begin{proof}
Corollary~\ref{wPJN4oayEM} and \eqref{ss4cNJda4m} imply that for each $j \in V$ there exists $\eta \in X_{j}(S)$ such that $y_j \in J_{\eta}$ .
Fix $j \in V$, $\xi \in X^{n}(S)$ with $\bt(\xi)=j$ and $U \in \Comp(\xi, R)$.
Then, Proposition \ref{HSecY02Qdj} (ii) implies that  
$\gamma_{\xi, U}(y_j) \in g_{\xi}^{-1}(J_{\eta})=J_{\xi\eta}$ and thus $(\xi\eta,\gamma_{\xi, U}(y_j)) \in J(\tilde{f})$.
Since the map $\tilde{f}$ is expanding along fibres, we have 
\[
|(g_{\xi})^{\prime}(\gamma_{\xi, U}(y_j))|=|(\tilde{f}^{n})^{\prime}(\xi\eta,\gamma_{\xi, U}(y_j))|\ \ge\ c \lambda^n.
\]
Therefore,
\[
\mathrm{w}_{\delta}(\xi, U)=|(g_{\xi})^{\prime}(\gamma_{\xi, U}(y_j))|^{-\delta}\leq c^{-\delta}\lambda^{-\delta n}.
\]
Since $\#\Comp(\xi, R) = \deg(g_{\xi})$, we obtain
\[
Z^{\mathrm{geom}}_{n} (\delta, R)
=
\sum_{\xi \in X^{n}(S)} \sum_{U\in \Comp(\xi, R)}
\mathrm{w}_{\delta}(\xi, U) 
\le
\sum_{\xi \in X^{n}(S)}\deg(g_{\xi}) c^{-\delta}\lambda^{-\delta n}
=
c^{-\delta}\lambda^{-\delta n}N_n,
\]
which implies 
$\overline{\Press^{\mathrm{geom}}}(\delta, R)\leq \Press(\sigma, \psi)-\delta\log\lambda<\Press(\sigma, \psi)$ and completes the proof.
\end{proof}

\begin{thm}\label{RspvxUCucI}
Let $S=(V, E, (\Gamma_{e})_{e \in E})$ be a finitely generated, irreducible rational GDMS.
Assume that $S$ satisfies the VSC and that the map $\tilde{f}$ is expanding along fibres
and topologically transitive on $J(\tilde{f})$.
Take an arbitrary family of vertex-wise holes $(\sfH_{i}(R))_{i \in V}$ for $S$.
Then, we have 
\[
e_\delta(R)=\Press(\sigma, \psi)-\overline{\Press^{\mathrm{geom}}}(\delta, R) > 0.
\]
\end{thm}
 
\begin{proof}
For each $i \in V$ set
\[
\mathcal{B}_{n,i}(R):=\bigcup_{\xi \in X_{i}^{n}(S)} \bigcup_{U \in \Comp(\xi, R)} B_{\xi, U}(R).
\]
Then, we have  $\sfHP_{n}(R)=\bigcup_{i \in V}\mathcal{B}_{n,i}(R)$.
By the iterated VSC, we obtain
\[
\nu_{\delta}(\mathcal{B}_{n,i}(R))=\sum_{\xi \in X_{i}^{n}(S)} \sum_{U \in \Comp(\xi, R)}\nu_{\delta}(B_{\xi, U}(R)).
\]
Since $\nu_{\delta}(\sfHP_{n}(R))\ \ge\ \max_{i \in V}\nu_{\delta} \left(\mathcal{B}_{n,i}(R)\right)$ and 
\[
\max_{i \in V}\nu_{\delta} \left(\mathcal{B}_{n,i}(R)\right)
\ \ge\
\frac1{\#V}\sum_{i \in V}\nu_{\delta} \left(\mathcal{B}_{n,i}(R)\right),
\]
we have 
\[
\frac1{\#V}\sum_{i \in V}\nu_{\delta}(\mathcal{B}_{n,i}(R))
 \leq 
\nu_{\delta}(\sfHP_{n}(R)).
\]
Therefore, we have 
\[
\frac1{\#V}\sum_{i \in V}\nu_{\delta}(\mathcal{B}_{n,i}(R))
 \leq 
\nu_{\delta}(\sfHP_{n}(R))
 \leq 
\sum_{i \in V}\nu_{\delta}(\mathcal{B}_{n,i}(R)).
\]
Therefore,
\[
\frac1{\#V}\sum_{\xi \in X^{n}(S)} \sum_{U \in \Comp(\xi, R)}\nu_{\delta}(B_{\xi, U}(R))
 \leq 
\nu_{\delta}(\sfHP_{n}(R))
 \leq 
\sum_{\xi \in X^{n}(S)} \sum_{U \in \Comp(\xi, R)}\nu_{\delta}(B_{\xi, U}(R)).
\]
Applying Lemma \ref{Ra4DgHuaXm} yields constants $c_*,C_*>0$ such that
\[
c_* Z_n^{\mathrm{geom}}(\delta, R) \leq \nu_{\delta}(\sfHP_{n}(R)) \leq C_* Z_n^{\mathrm{geom}}(\delta, R)
\]
for each $n \in \N$.
Therefore,
\[
\limsup_{n \to \infty}\frac{1}{n}\log \nu_{\delta}(\sfHP_{n}(R))
=
\limsup_{n \to \infty}\frac{1}{n}\log Z_n^{\mathrm{geom}}(\delta, R)
=
\overline{\Press^{\mathrm{geom}}}(\delta, R).
\]
and thus
\[
e_\delta(R)
=
-\limsup_{n \to \infty}\frac{1}{n}\log\frac{\nu_{\delta}(\sfHP_{n}(R))}{N_n}
=
-\overline{\Press^{\mathrm{geom}}}(\delta, R)+\Press(\sigma, \psi).
\]
Combining this with Proposition \ref{grxaIp1jSL} completes the proof.
\end{proof}

\end{document}